\RequirePackage{etex}
\RequirePackage{easybmat}
\documentclass[]{article}
\usepackage{minitoc}
\usepackage[titletoc]{appendix}
\usepackage[utf8]{inputenc}
\usepackage[T1]{fontenc}
\usepackage{indentfirst}
\usepackage{graphicx}
\usepackage{amsmath}
\usepackage{listings}
\usepackage{moreverb}
\usepackage{eurosym}
\usepackage{fancyvrb}
\usepackage{afterpage}
\usepackage[final]{pdfpages}
\usepackage[nottoc]{tocbibind}
\usepackage{listings}
\usepackage{amsthm}
\usepackage{amsfonts}
\usepackage{geometry}
\usepackage{xypic}
\xyoption{all}
\xyoption{arc}
\xyoption{poly}
\usepackage[]{algorithm2e}
\usepackage{hyperref}
\usepackage{fixltx2e}
\usepackage{amssymb}
\usepackage{tikz}
\usepackage{tkz-graph}
\usetikzlibrary{calc}
\usepackage{float}
\usepackage{color}
\usepackage{lipsum}

\newcommand*{\mathcolor}{}
\def\mathcolor#1#{\mathcoloraux{#1}}

\newtheorem{theorem}{Theorem}[section]
\newtheorem{lemma}[theorem]{Lemma}
\newtheorem{proposition}[theorem]{Proposition}

\newtheorem{definition}[theorem]{Definition}
\newtheorem{example}[theorem]{Example}

\newtheorem{remark}[theorem]{Remark}

\newtheorem{convention}[theorem]{Convention}

\makeatletter
\def\v@rt#1#2{\m@th\ooalign{$\hfil#1|\hfil$\crcr$#1#2$}}
\def\captr{\mathrel{\mathpalette\v@rt\cap}}
\makeatother

\makeatletter

\newcommand{\Rmnum}[1]{\expandafter\@slowromancap\romannumeral #1@}
\makeatother

\begin{document}

\title{Geodesic Normal Forms and Hecke Algebras for the Complex Reflection Groups $G(de,e,n)$}
\author{Georges Neaime\\
Fakultät für Mathematik\\
Universität Bielefeld\\
\texttt{gneaime@math.uni-bielefeld.de}
}

\date{October 22, 2018}

\maketitle

\begin{abstract}

We establish geodesic normal forms for the general series of complex reflection groups $G(de,e,n)$ by using the presentations of Corran-Picantin and Corran-Lee-Lee of $G(e,e,n)$ and $G(de,e,n)$ for $d > 1$, respectively. This requires the elaboration of a combinatorial technique in order to explicitly determine minimal word representatives of the elements of $G(de,e,n)$. Using these geodesic normal forms, we construct natural bases for the Hecke algebras associated with the complex reflection groups $G(e,e,n)$ and $G(d,1,n)$. As an application, we obtain a new proof of the BMR freeness conjecture for these groups.

\end{abstract}

\vspace{0.5cm}

\tableofcontents 

\newpage

\section{Introduction}\label{SectionIntroduction}

Complex reflection groups are finite groups generated by complex reflections. Recall that a complex reflection is a linear transformation of finite order that fixes a hyperplane pointwise. These groups include finite real relection groups, also known as finite Coxeter groups. It is well known that every complex reflection group is a direct product of irreducible ones. The irreducible complex reflection groups have been classified by Shephard and Todd \cite{ShephardTodd} in 1954. The classification includes the general $3$-parameter series $G(de,e,n)$ that can be easily described in terms of monomial matrices and $34$ exceptional groups denoted by $G_4$, $G_5$, $\cdots$, $G_{37}$. 

Brou\'e, Malle and Rouquier \cite{BMR} managed to attach a complex braid group to each complex reflection group. This generalizes the notion of Artin-Tits groups attached to finite Coxeter groups. Extending earlier results in \cite{BroueMalleZyklotomischeHeckealgebren}, they also managed to generalize the definition of the (Iwahori-)Hecke algebra for real reflection groups to arbitrary complex reflection groups by using their definition of the complex braid group. Actually, the Hecke algebra is defined as a quotient of the complex braid group algebra by some polynomial relations. It is believed that nice properties of these objects in the case of real reflection groups could be extended to the general case of complex reflection groups. In \cite{BMR}, it was stated a number of important conjectures about the complex braid groups and the Hecke algebras.

One of these interesting conjectures is the so-called ``BMR freeness conjecture''. It states that the Hecke algebra is a free module of rank equal to the order of the associated complex reflection group. This property is valid for the (Iwahori-)Hecke algebra attached to any finite real reflection group (see \cite{Bourbaki}), where a basis is constructed from geodesic normal forms in the finite Coxeter group due to Matsumoto's property (see \cite{MatsumotoTheorem}). The BMR freeness conjecture can be easily reduced to the case of irreducible complex reflection groups. During the past two decades, a proof of this conjecture for each case of the classification of Shephard and Todd has been established involving the results of a number of authors. As we are interested in the case of the general series of complex reflection groups, we mention that this conjecture has been established for $G(d,1,n)$ (see Ariki-Koike \cite{ArikiKoikeHecke} and Bremke-Malle \cite{BremkeMalle}) and for $G(de,e,n)$ by Ariki (see \cite{ArikiHecke} and Appendix A.2 of \cite{RostamArikiBMR}). A list of references for the proof of the BMR freeness conjecture can be found in the next section.

An important constraint in the proof of this conjecture, and in the theory of Hecke algebras for complex reflection groups in general, is the failure of an analogue of Matsumoto's property. That is, we do not have any canonical basis of the Hecke algebra. The proof of the BMR freeness conjecture was obtained by sometimes tedious and lengthy computations in order to explicitly construct bases for the Hecke algebras. It is then of importance to find nice bases for these algebras.  In this paper, we construct new and natural bases for the Hecke algebras attached to the complex reflection groups $G(e,e,n)$ and $G(d,1,n)$.

In order to establish these bases, our attention is firstly shifted to the complex reflection groups $G(de,e,n)$. We construct geodesic normal forms for these groups by using the presentations of Corran-Picantin \cite{CorranPicantin} and Corran-Lee-Lee \cite{CorranLeeLee} of $G(e,e,n)$ and $G(de,e,n)$ for $d > 1$, respectively. The geodesic normal forms are very natural and easy to describe. They generalize our construction in \cite{GeorgesNeaimeIntervals} for $G(e,e,n)$ to all the cases of the general series of complex reflection groups.

We establish that these geodesic normal forms provide bases for the Hecke algebras attached to $G(e,e,n)$ and $G(d,1,n)$. Since these bases are constructed from geodesic normal forms in the complex reflection groups, they are natural bases for the associated Hecke algebras. Note that the geodesic normal forms for $G(e,e,n)$ have been already used in our previous work \cite{GeorgesNeaimeIntervals} in order to construct intervals in $G(e,e,n)$ that give rise to nice combinatorial structures (called interval Garside structures) for the associated complex braid groups.

The article is organized as follows. In Section 2, we provide a basic background material and recall the BMR freeness conjecture. The geodesic normal forms for the complex reflection groups $G(de,e,n)$ are constructed in Section 3, which sets the stage for our later work. The techniques used are elementary and the associated combinatorial characterizations are very explicit. In Section 4, the attention shifts to the Hecke algebras. Actually, we establish nice presentations for the Hecke algebras attached to $G(de,e,n)$, by using the presentations of Corran-Picantin and Corran-Lee-Lee of the associated complex braid groups that we recall in the same section. The remaining part of the article establishes that the geodesic normal forms constructed in Section 3 provide natural bases for the Hecke algebras associated with $G(e,e,n)$ and $G(d,1,n)$. Through these constructions, we also obtain a new proof of the BMR freeness conjecture for these groups. At the end of Section 5, we explain why we were not able to establish such natural bases for the Hecke algebras attached to the other cases of the general series of complex reflection groups.

\section{Definitions and preliminaries}

Let $W$ be a complex reflection group. Let $\mathcal{A} := \{Ker(s-1)$ s.t. $s \in \mathcal{R} \}$ be the hyperplane arrangement and $X := \mathbb{C}^n\setminus\bigcup\mathcal{A}$ be the hyperplane complement. The complex reflection group $W$ acts naturally on $X$. Let $X/W$ be its space of orbits. The complex braid group $B$ attached to $W$ is defined as follows. For details about this definition, we refer to \cite{BMR}.
\begin{definition}
The complex braid group attached to $W$ is the fundamental group  $$B:=\pi_1(X/W).$$
\end{definition}


Recall that a complex reflection $s \in W$ is called distinguished if its only nontrivial eigenvalue is $exp(2i\pi / o(s))$, where $o(s)$ denotes the order of $s$. For the standard notion of braided reflections that we use in the next definition, the reader may check \cite{BMR}. We are ready to define the Hecke algebra associated with $W$ (see \cite{BMR} and \cite{MarinG20G21}).

\begin{definition}\label{HeckeBMR}
Let $R=\mathbb{Z}[a_{s,i},a_{s,0}^{-1}, 0 \leq i \leq o(s)-1]$, where s runs over the distinguished reflections, with the convention $a_{s,i} = a_{s',i}$ if $s$ and $s'$ are conjugates in $W$. The Hecke algebra $H(W)$ attached to the complex reflection group $W$ is the quotient of the complex braid group algebra $RB$ by the relations $$\sigma^{o(s)} -  a_{s, o(s)-1}\sigma^{o(s)-1} - \cdots - a_{s,0} = 0,$$ for each braided reflection $\sigma$ associated with $s$. 
\end{definition}

Note that it is enough to choose one such relation per conjugacy class of distinguished reflections, as all the corresponding braided reflections are conjugates in $B$ (see \cite{BMR}).\\

The BMR freeness conjecture proposed by Brou\'e, Malle and Rouquier \cite{BMR} in 1998 states that the Hecke algebra $H(W)$ attached to $W$ is a free $R$-module of rank equal to the order of $W$. After two decades, the BMR freeness conjecture is proven through the results of a number of authors. Thus, we have the following theorem.

\begin{theorem}\label{PropositionBMRFreenessTheorem}

The Hecke algebra $H(W)$ is a free $R$-module of rank $|W|$.

\end{theorem}

The BMR freeness conjecture can be easily reduced to the case where $W$ is irreducible. It is true for the (Iwahori-)Hecke algebra attached to any finite Coxeter group (see Lemma 4.4.3 of \cite{GeckPfeifferBook}). Ariki and Koike \cite{ArikiKoikeHecke} proved it for the case of $G(d,1,n)$. Note that a basis for the Hecke algebra associated with $G(d,1,n)$ is also given in \cite{BremkeMalle}. Ariki defined in \cite{ArikiHecke} a Hecke algebra for $G(de,e,n)$ by a presentation with generators and relations. He also proved that it is a free module of rank $|G(de,e,n)|$. The Hecke algebra defined by Ariki is isomorphic to the Hecke algebra defined by Brou\'e, Malle, and Rouquier in \cite{BMR} for $G(de,e,n)$. The details why this is true can be found in Appendix A.2 of \cite{RostamArikiBMR}. Hence one gets a proof of Theorem \ref{PropositionBMRFreenessTheorem} for the general series of complex reflection groups.

Concerning the exceptional complex reflection groups, Marin proved the conjecture for $G_4$, $G_{25}$, $G_{26}$, and $G_{32}$ in \cite{MarinCubicHeckeAlgebra5Strands} and \cite{MarinBMRFreenessConjecture}. Marin and Pfeiffer proved it for $G_{12}$, $G_{22}$, $G_{24}$, $G_{27}$, $G_{29}$, $G_{31}$, $G_{33}$, and $G_{34}$ in \cite{MarinPfeifferBMRFreeness}. In her PhD thesis and in the article that followed (see \cite{ChavliThesis} and \cite{ChavliArticlePhD}), Chavli proved the validity of this conjecture for $G_{5}$, $G_{6}$, $\cdots$, $G_{16}$. Recently, Marin proved the conjecture for $G_{20}$ and $G_{21}$ (see \cite{MarinG20G21}) and finally Tsushioka for $G_{17}$, $G_{18}$ and $G_{19}$ (see \cite{TsuchiokaBMRfreenessConjecture}). Hence we obtain a proof of Theorem \ref{PropositionBMRFreenessTheorem} for all the cases of irreducible complex reflection groups.

\section{Geodesic normal forms for \texorpdfstring{$G(de,e,n)$}{TEXT}}\label{SectionReducedWordsG(de,e,n)}

This section sets the stage for our later work by establishing a set of geodesic normal forms for the complex reflection groups $G(de,e,n)$, using the generating sets introduced by Corran-Picantin \cite{CorranPicantin} and Corran-Lee-Lee \cite{CorranLeeLee} for $G(e,e,n)$ and $G(de,e,n)$ for $d > 1$, respectively. The case of $G(e,e,n)$ has been already done in our previous work (see Section 3 of \cite{GeorgesNeaimeIntervals}). We generalize the combinatorial techniques used there to the case of $G(de,e,n)$ for $d > 1$. We obtain natural and explicit geodesic normal forms for these groups.

\subsection{Presentations for \texorpdfstring{$G(de,e,n)$}{TEXT}}

The complex reflection group $G(de,e,n)$ is the group of monomial matrices whose nonzero entries are $de$-th roots of unity and their product is a $d$-th root of unity.

Set $d=1$ and let $e \geq 1$ and $n \geq 2$. Corran-Picantin discovered in \cite{CorranPicantin} a presentation of the complex reflection group $G(e,e,n)$ that is defined as follows.

\begin{definition}\label{DefPresCorranPicantinG(e,e,n)}

The complex reflection group $G(e,e,n)$ is defined by a presentation with generators: $\{\textbf{t}_i \ |\ i \in \mathbb{Z}/e\mathbb{Z}\} \cup \{\textbf{s}_3, \textbf{s}_4, \cdots, \textbf{s}_n \}$ and relations:

\begin{enumerate}

\item $\textbf{t}_i \textbf{t}_{i-1} = \textbf{t}_j \textbf{t}_{j-1}$ for $i, j \in \mathbb{Z}/e\mathbb{Z}$,
\item $\textbf{t}_i \textbf{s}_3 \textbf{t}_i = \textbf{s}_3 \textbf{t}_i \textbf{s}_3$ for $i \in \mathbb{Z}/e\mathbb{Z}$,
\item $\textbf{s}_j \textbf{t}_i = \textbf{t}_i \textbf{s}_j$ for $i \in \mathbb{Z}/e\mathbb{Z}$ and $4 \leq j \leq n$,
\item $\textbf{s}_i \textbf{s}_{i+1} \textbf{s}_i = \textbf{s}_{i+1} \textbf{s}_i \textbf{s}_{i+1}$ for $3 \leq i \leq n-1$, 
\item $\textbf{s}_i \textbf{s}_j = \textbf{s}_j \textbf{s}_i$ for $|i-j| > 1$,
\item $\textbf{t}_i^2=1$ for $i \in \mathbb{Z}/e\mathbb{Z}$ and $\textbf{s}_j^2=1$ for $3 \leq j \leq n$.

\end{enumerate}

\end{definition}

The matrices in $G(e,e,n)$ that correspond to the generators of this presentation are given by  $\mathbf{t}_i \longmapsto t_i:= \begin{pmatrix}

0 & \zeta_{e}^{-i} & 0\\
\zeta_{e}^{i} & 0 & 0\\
0 & 0 & I_{n-2}\\

\end{pmatrix}$ \mbox{for $0 \leq i \leq e-1$}, and $\mathbf{s}_j \longmapsto s_j:= \begin{pmatrix}

I_{j-2} & 0 & 0 & 0\\
0 & 0 & 1 & 0\\
0 & 1 & 0 & 0\\
0 & 0 & 0 & I_{n-j}\\

\end{pmatrix}$ for $3 \leq j \leq n$. To avoid confusion, we use regular letters for matrices and bold letters for words over the generating set of the presentation of Corran-Picantin.\\

This presentation can be described by the following diagram. It is interpreted, where possible, as a Coxeter diagram. In addition, the dashed circle describes Relation $1$ of Definition \ref{DefPresCorranPicantinG(e,e,n)}.

\begin{figure}[H]
\begin{center}
\begin{tikzpicture}[yscale=0.8,xscale=1,rotate=30]

\draw[thick,dashed] (0,0) ellipse (2cm and 1cm);

\node[draw, shape=circle, fill=white, label=above:$\mathbf{t}_0$] (t0) at (0,-1) {\begin{tiny} 2 \end{tiny}};
\node[draw, shape=circle, fill=white, label=above:$\mathbf{t}_1$] (t1) at (1,-0.8) {\begin{tiny} 2 \end{tiny}};
\node[draw, shape=circle, fill=white, label=right:$\mathbf{t}_2$] (t2) at (2,0) {\begin{tiny} 2 \end{tiny}};
\node[draw, shape=circle, fill=white, label=above:$\mathbf{t}_i$] (ti) at (0,1) {\begin{tiny} 2 \end{tiny}};
\node[draw, shape=circle, fill=white, label=above:$\mathbf{t}_{e-1}$] (te-1) at (-1,-0.8) {\begin{tiny} 2 \end{tiny}};

\draw[thick,-] (0,-2) arc (-180:-90:3);

\node[draw, shape=circle, fill=white, label=below left:$\mathbf{s}_3$] (s3) at (0,-2) {\begin{tiny} 2 \end{tiny}};

\draw[thick,-] (t0) to (s3);
\draw[thick,-,bend left] (t1) to (s3);
\draw[thick,-,bend left] (t2) to (s3);
\draw[thick,-,bend left] (s3) to (te-1);

\node[draw, shape=circle, fill=white, label=below:$\mathbf{s}_4$] (s4) at (0.15,-3) {\begin{tiny} 2 \end{tiny}};
\node[draw, shape=circle, fill=white, label=below:$\mathbf{s}_{n-1}$] (sn-1) at (2.2,-4.9) {\begin{tiny} 2 \end{tiny}};
\node[draw, shape=circle, fill=white, label=right:$\mathbf{s}_{n}$] (sn) at (3,-5) {\begin{tiny} 2 \end{tiny}};

\node[fill=white] () at (1,-4.285) {$\cdots$};
\end{tikzpicture}
\end{center}
\caption{Diagram for the presentation of Corran-Picantin of $G(e,e,n)$.}\label{FigureDiagramCPG(e,e,n)}
\end{figure}
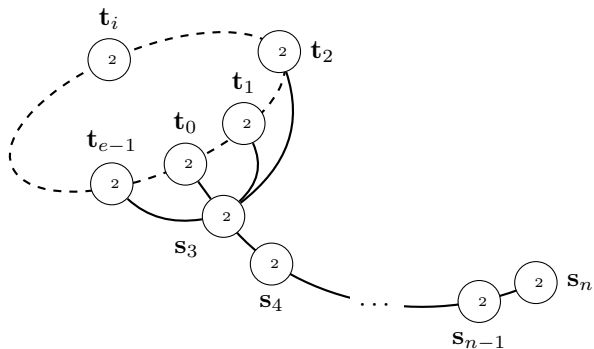

\begin{remark}

\begin{enumerate}

\item For $e=1$ and $n \geq 2$, we obtain the classical presentation of the symmetric group $S_n$.
\item For $e=2$ and $n \geq 2$, we obtain the classical presentation of the Coxeter group of type $D_n$.

\end{enumerate}

\end{remark}

The attention shifts now to the case $d > 1$. Let $d> 1$, $e \geq 1$ and $n \geq 2$. There exists a presentation of the complex reflection group $G(de,e,n)$ discovered by Corran-Lee-Lee in \cite{CorranLeeLee}.

\begin{definition}\label{DefCorranLeePresG(de,e,n)}

The complex reflection group $G(de,e,n)$ is defined by a presentation with set of generators: $\mathbf{X} = \{ \mathbf{z}\} \cup \{ \mathbf{t}_i \ | \ i \in \mathbb{Z}/de\mathbb{Z}\} \cup \{\mathbf{s}_3, \mathbf{s}_4, \cdots, \mathbf{s}_n \}$ and relations as follows.

\begin{enumerate}

\item $\mathbf{z} \mathbf{t}_i = \mathbf{t}_{i-e} \mathbf{z}$ for $i \in \mathbb{Z}/de\mathbb{Z}$,
\item $\mathbf{z}  \mathbf{s}_j=\mathbf{s}_j \mathbf{z}$ for $3 \leq j \leq n$,
\item $\mathbf{t}_i \mathbf{t}_{i-1} = \mathbf{t}_j \mathbf{t}_{j-1}$ for $i, j \in \mathbb{Z}/de\mathbb{Z}$,
\item $\mathbf{t}_i \mathbf{s}_3 \mathbf{t}_i = \mathbf{s}_3 \mathbf{t}_i \mathbf{s}_3$ for $i \in \mathbb{Z}/de\mathbb{Z}$,
\item $\mathbf{s}_j \mathbf{t}_i = \mathbf{t}_i \mathbf{s}_j$ for $i \in \mathbb{Z}/de\mathbb{Z}$ and $4 \leq j \leq n$,
\item $\mathbf{s}_i \mathbf{s}_{i+1} \mathbf{s}_i = \mathbf{s}_{i+1} \mathbf{s}_i \mathbf{s}_{i+1}$ for $3 \leq i \leq n-1$,
\item $\mathbf{s}_i \mathbf{s}_j = \mathbf{s}_j \mathbf{s}_i$ for $|i-j| > 1$, 
\item $\mathbf{z}^d=1$, $\mathbf{t}_i^2=1$ for $i \in \mathbb{Z}/de\mathbb{Z}$, and $\mathbf{s}_j^2=1$ for $3 \leq j \leq n$.

\end{enumerate}

\end{definition}

The generators of this presentation correspond to the following $n \times n$ matrices.
The generator $\mathbf{t}_i$ is represented by the matrix $t_i =
\begin{pmatrix}

0 & \zeta_{de}^{-i} & 0\\
\zeta_{de}^{i} & 0 & 0\\
0 & 0 & I_{n-2}\\

\end{pmatrix}$
for $i \in \mathbb{Z}/de\mathbb{Z}$, $\mathbf{z}$ by the diagonal matrix  $z = Diag(\zeta_d,1,\cdots,1)$ where 
$\zeta_d = exp(2i \pi / d)$, 
and $\mathbf{s}_j$ by the transposition matrix $s_j = (j-1,j)$ for $3 \leq j \leq n$. To avoid confusion, we use regular letters for matrices and bold letters for words over $\mathbf{X}$. Denote by $X$ the set $\{z,t_0,t_1, \cdots, t_{de-1},s_3, \cdots, s_n\}$.\\

This presentation can be described by the following diagram. This diagram is again to be read, where possible, as a Coxeter diagram. In addition, the dashed circle describes Relation $3$ of Definition \ref{DefCorranLeePresG(de,e,n)} and the curved arrow above the node corresponding to $\mathbf{z}$ describes Relation $1$.

\begin{figure}[H]
\begin{center}
\begin{tikzpicture}[yscale=0.8,xscale=1,rotate=30]

\draw[thick,dashed] (0,0) ellipse (2cm and 1cm);

\node[draw, shape=circle, fill=white, label=above:\begin{small}$\mathbf{z}$\end{small}] (z) at (0,0) {\begin{tiny}$d$\end{tiny}};
\node[draw, shape=circle, fill=white, label=above:\begin{small}$\mathbf{t}_0$\end{small}] (t0) at (0,-1) {\begin{tiny}$2$\end{tiny}};
\node[draw, shape=circle, fill=white, label=above:\begin{small}$\mathbf{t}_1$\end{small}] (t1) at (1,-0.8) {\begin{tiny}$2$\end{tiny}};
\node[draw, shape=circle, fill=white, label=right:\begin{small}$\mathbf{t}_2$\end{small}] (t2) at (2,0) {\begin{tiny}$2$\end{tiny}};
\node[draw, shape=circle, fill=white, label=above:$\mathbf{t}_i$] (ti) at (0,1) {\begin{tiny}$2$\end{tiny}};
\node[draw, shape=circle, fill=white, label=above:\begin{small}$\mathbf{t}_{de-1}$\end{small}] (te-1) at (-1,-0.8) {\begin{tiny}$2$\end{tiny}};

\draw[->,bend left=45] (-0.3,0.26) to (0.3,0.26);
\draw[thick,-] (0,-2) arc (-180:-90:3);

\node[draw, shape=circle, fill=white, label=below left:$\mathbf{s}_3$] (s3) at (0,-2) {\begin{tiny}$2$\end{tiny}};

\draw[thick,-] (t0) to (s3);
\draw[thick,-,bend left] (t1) to (s3);
\draw[thick,-,bend left] (t2) to (s3);
\draw[thick,-,bend left] (s3) to (te-1);

\node[draw, shape=circle, fill=white, label=below:$\mathbf{s}_4$] (s4) at (0.15,-3) {\begin{tiny}$2$\end{tiny}};
\node[draw, shape=circle, fill=white, label=below:$\mathbf{s}_{n-1}$] (sn-1) at (2.2,-4.9) {\begin{tiny}$2$\end{tiny}};
\node[draw, shape=circle, fill=white, label=right:$\mathbf{s}_{n}$] (sn) at (3,-5) {\begin{tiny}$2$\end{tiny}};

\node[fill=white] () at (1,-4.285) {$\cdots$};

\end{tikzpicture}
\end{center}
\caption{\mbox{Diagram for the presentation of Corran-Lee-Lee of $G(de,e,n)$.}}\label{FigPresG(de,e,n)CorranLeeLee}
\end{figure}

\begin{proposition}\label{PropPresB(d,1,n)}

Let $e=1$. The presentation given in Definition \ref{DefCorranLeePresG(de,e,n)} is equivalent to the classical presentation of the complex reflection group $G(d,1,n)$ that can be described by the following well-known diagram.

\begin{figure}[H]
\begin{small}
\begin{center}
\begin{tikzpicture}

\node[draw, shape=circle, label=above:$\mathbf{z}$] (1) at (0,0) {$d$};
\node[draw, shape=circle,label=above:$\mathbf{s}_2$] (2) at (2,0) {$2$};
\node[draw, shape=circle,label=above:$\mathbf{s}_3$] (3) at (4,0) {$2$};
\node[draw, shape=circle,label=above:$\mathbf{s}_{n-1}$] (n-1) at (6,0) {$2$};
\node[draw,shape=circle,label=above:$\mathbf{s}_n$] (n) at (8,0) {$2$};

\draw[thick,-,double] (1) to (2);
\draw[thick,-] (2) to (3);
\draw[dashed,-,thick] (3) to (n-1);
\draw[thick,-] (n-1) to (n);

\end{tikzpicture}
\end{center}
\end{small}

\caption{Diagram for the presentation of $G(d,1,n)$.}\label{DiagPresofBMRofGd1n}
\end{figure}
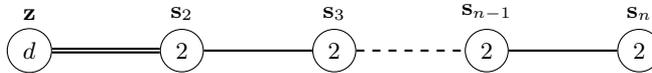

\end{proposition}

\begin{proof}

Let $e=1$. Relation $1$ of Definition \ref{DefCorranLeePresG(de,e,n)} becomes $\mathbf{z} \mathbf{t}_1=\mathbf{t}_0 \mathbf{z}$, that is $\mathbf{t}_1 = \mathbf{z}^{-1} \mathbf{t}_0 \mathbf{z}$. Also by Relation $3$ of  Definition \ref{DefCorranLeePresG(de,e,n)}, we have $\mathbf{t}_k=\mathbf{z}^{-k}\mathbf{t}_0\mathbf{z}^k$ for $1 \leq k \leq d-1$. If we remove $\mathbf{t}_1, \cdots, \mathbf{t}_{d-1}$ from the set of generators and replace every occurrence of $\mathbf{t}_k$ in the defining relations with $\mathbf{z}^{-k} \mathbf{t}_0 \mathbf{z}^{k}$ for $1 \leq k \leq d-1$, we recover the classical presentation of the complex reflection group $G(d,1,n)$. Note that we replace $\mathbf{t}_0$ by $\mathbf{s}_2$ in the set of generators of this presentation.

\end{proof}

\begin{remark}

For $d = 2$, the presentation described by the diagram of Figure \ref{DiagPresofBMRofGd1n} corresponds to the presentation of the Coxeter group of type $B_n$. 

\end{remark}

From now on, we set the following convention.

\begin{convention}\label{ConventionDecreaseIncreaseIndex}

A decreasing-index expression of the form $\mathbf{s}_i \mathbf{s}_{i-1} \cdots \mathbf{s}_{i'}$ is the empty word when $i < i'$ and an increasing-index expression of the form $\mathbf{s}_i \mathbf{s}_{i+1} \cdots \mathbf{s}_{i'}$ is the empty word when $i > i'$. Similarly, in $G(de,e,n)$, a decreasing-index product of the form $s_i s_{i-1} \cdots s_{i'}$ is equal to $I_n$ when $i < i'$ and an increasing-index product of the form $s_i s_{i+1} \cdots s_{i'}$ is equal to $I_n$ when $i > i'$, where $I_n$ is the identity $n \times n$ matrix. We also set that $\mathbf{z}^0$ is the empty word.

\end{convention}

\subsection{Minimal word representatives}

Consider the complex reflection group $G(de,e,n)$ with $d > 1$, $e > 1$ and $n \geq 2$. The case of $G(d,1,n)$ will be established in the next subsection. Recall that $\mathbf{X}$ denotes the set of the generators of the presentation of Corran-Lee-Lee of $G(de,e,n)$ and $X$ the set of the corresponding matrices. Our aim is to represent each element of $G(de,e,n)$ by a reduced word over $\mathbf{X}$. This requires the elaboration of a combinatorial technique in order to determine a reduced expression decomposition over $\mathbf{X}$ of an element of $G(de,e,n)$.\\

We introduce Algorithm \ref{Algo2} below (see page 8) that produces a word $R\!E(w)$ over $\mathbf{X}$ for a given element $w$ in $G(de,e,n)$. This algorithm generalizes the one introduced in Section 3 of our previous work \cite{GeorgesNeaimeIntervals} that corresponds to the case of $G(e,e,n)$. Note that we use Convention \ref{ConventionDecreaseIncreaseIndex} in the elaboration of the algorithm. Later on, we prove that its output $R\!E(w)$ is a reduced expression over $\mathbf{X}$ of $w \in G(de,e,n)$.\\

Let $w_n := w \in G(de,e,n)$. For $i$ from $n$ to $2$, the $i$-th step of \mbox{Algorithm\! \ref{Algo2}} transforms the block diagonal matrix $\left(
\begin{array}{c|c}
w_i & 0 \\
\hline
0 & I_{n-i}
\end{array}
\right)$ into a block diagonal matrix $\left(
\begin{array}{c|c}
w_{i-1} & 0 \\
\hline
0 & I_{n-i+1}
\end{array}
\right) \in G(de,e,n)$. Actually, for $2 \leq i \leq n$, there exists a unique $c$ with $1 \leq c \leq n$ such that $w_i[i,c] \neq 0$. At each step $i$ of Algorithm\! \ref{Algo2}, if $w_i[i,c] =1$, we shift it into the diagonal position $[i,i]$ by right multiplication by transpositions. If $w_i[i,c] \neq 1$, we shift it into the first column by right multiplication by transpositions, transform it into $1$ by right multiplication by an element of $\{t_0, t_1, \cdots, t_{de-1} \}$, and then shift the $1$ obtained into the diagonal position $[i,i]$. Finally, we get $w_1 = \zeta_d^k$ for some $0 \leq k \leq d-1$, where $\zeta_d^k$ is equal to the product of the nonzero entries of $w$. By multiplying $\left(
\begin{array}{c|c}
w_{1} & 0 \\
\hline
0 & I_{n-1}
\end{array}
\right)$ on the right by $z^{-k}$, we get the identity matrix $I_n$.

\begin{algorithm}

\SetKwInOut{Input}{Input}\SetKwInOut{Output}{Output}

\noindent\rule{12cm}{0.5pt}

\Input{$w$, a matrix in $G(de,e,n)$ with $d > 1$, $e > 1$, and $n \geq 2$.}
\Output{$R\!E(w)$, a word over $\mathbf{X}$.}

\noindent\rule{12cm}{0.5pt}

\textbf{Local variables}: $w'$, $R\!E(w)$,  $i$, $U$, $V$, $c$, $k$.

\noindent\rule{12cm}{0.5pt}

\textbf{Initialisation}:
$U:=[1,\zeta_{de},\zeta_{de}^2,...,\zeta_{de}^{e-1}]$, $V:=[1,\zeta_d,\zeta_d^2, \cdots, \zeta_d^{d-1}]$, $s_2 := t_0$, $\mathbf{s}_2 := \mathbf{t}_0$, $R\!E(w) := \varepsilon$: the empty word, $w' := w$.

\noindent\rule{12cm}{0.5pt}

\For{$i$ \textbf{from} $n$ \textbf{down to} $2$} {
	$c:=1$; $k:=0$; \\
	\While{$w'[i,c] = 0$}{$c :=c+1$\; 
	}
	 \textit{\#Then $w'[i,c]$ is the root of unity on the row $i$}\;
	\While{$U[k+1]\neq w'[i,c] $}{$k :=k+1$\;
	}
	\textit{\#Then $w'[i,c] = \zeta_{de}^k$.}\\

		\If{$k \neq 0$}{
		$w' := w' s_{c} s_{c-1} \cdots s_{3} s_{2} t_{k}$; \textit{\#Then $w'[i,2] =1$}\;
		$R\!E(w) := \mathbf{t}_k \mathbf{s}_2 \mathbf{s}_3 \cdots \mathbf{s}_c R\!E(w)$\;
		$c:=2$\;
	}	
	$w' := w' s_{c+1} \cdots s_{i-1}s_{i}$; \textit{\#Then $w'[i,i] = 1$}\;
	$R\!E(w) := \mathbf{s}_i \mathbf{s}_{i-1} \cdots \mathbf{s}_{c+1} R\!E(w)$\;
}
$k:=0$\;
\While {$V[k+1] \neq w'[1,1]$}{$k:=k+1$\;}
\textit{\#Then $w'[1,1] = \zeta_d^k$}\;
$w':=w' z^{-k}$; \textit{\#Then $w'=I_n$}\;
\If{$k \neq 0$}{$R\!E(w) = \mathbf{z}^{k} R\!E(w)$\;}
\textbf{Return} $R\!E(w)$;

\noindent\rule{12cm}{0.5pt}

\caption{A word over $\mathbf{X}$ corresponding to a matrix $w \in G(de,e,n)$.}\label{Algo2}
\end{algorithm}

\begin{example}\label{ExampleAlgo1}

We apply Algorithm\! \ref{Algo2} to $w := \begin{pmatrix}

\zeta_9 & 0 & 0 & 0\\
0 & 0 & 1 & 0\\
0 & 0 & 0 & \zeta_9\\
0 & \zeta_9 & 0 & 0\\

\end{pmatrix}$ $\in G(9,3,4)$.\\
Step $1$ $(i=4, k=0, c=1)$: $w':= w s_2=\begin{pmatrix}

0 & \zeta_9 & 0 & 0\\
0 & 0 & 1 & 0\\
0 & 0 & 0 & \zeta_9\\
\zeta_9 & 0 & 0 & 0\\

\end{pmatrix}$, then $w':= w' t_1 = \begin{pmatrix}

\zeta_9^2 & 0 & 0 & 0\\
0 & 0 & 1 & 0\\
0 & 0 & 0 & \zeta_9\\
0 & 1 & 0 & 0\\

\end{pmatrix}$, then $w':= w' s_3 s_4 = \begin{pmatrix}

\zeta_9^2 & 0 & 0 & 0\\
0 & 1 & 0 & 0\\
0 & 0 & \boxed{\zeta_9} & 0\\
0 & 0 & 0 & \mathbf{1}\\

\end{pmatrix}$.\\
Step $2$ $(i=3, k=1, c=3)$: $w' := w' s_3 s_2 = \begin{pmatrix}

0 & \zeta_9^2 & 0 & 0\\
0 & 0 & 1 & 0\\
\zeta_9 & 0 & 0 & 0\\
0 & 0 & 0 & 1\\

\end{pmatrix}$, then $w' := w' t_1 = \begin{pmatrix}

\zeta_9^{3} & 0 & 0 & 0\\
0 & 0 & 1 & 0\\
0 & 1 & 0 & 0\\
0 & 0 & 0 & 1\\

\end{pmatrix}$, then $w':= w' s_3 = \begin{pmatrix}

\zeta_9^3 & 0 & 0 & 0\\
0 & \boxed{1} & 0 & 0\\
0 & 0 & \mathbf{1} & 0\\
0 & 0 & 0 & 1\\

\end{pmatrix}$.\\
Step $3$ $(i=2, k=0, c=2)$: $w' = \begin{pmatrix}

\boxed{\zeta_9^3} & 0 & 0 & 0\\
0 & \mathbf{1} & 0 & 0\\
0 & 0 & 1 & 0\\
0 & 0 & 0 & 1\\

\end{pmatrix} = \begin{pmatrix}

\boxed{\zeta_3} & 0 & 0 & 0\\
0 & \mathbf{1} & 0 & 0\\
0 & 0 & 1 & 0\\
0 & 0 & 0 & 1\\

\end{pmatrix}$.\\
Step $4$ $(k=1)$: $w' := w' z^{-1} = I_4$.\\
Hence $R\!E(w) = \mathbf{z} \mathbf{s}_3 \mathbf{t}_1 \mathbf{s}_2 \mathbf{s}_3 \mathbf{s}_4 \mathbf{s}_3 \mathbf{t}_1 \mathbf{s}_2 = \mathbf{z} \mathbf{s}_3 \mathbf{t}_1 \mathbf{t}_0 \mathbf{s}_3 \mathbf{s}_4 \mathbf{s}_3 \mathbf{t}_1 \mathbf{t}_0$ \emph{(since $\mathbf{s}_2 = \mathbf{t}_0$)}.

\end{example}

The next lemma follows immediately from Algorithm \ref{Algo2}. It explains how we can easily obtain the blocks defined in the algorithm.

\begin{lemma}\label{LemmaBlocksG(de,e,n)}

For $2 \leq i \leq n$, suppose $w_i[i,c] \neq 0$. The block $w_{i-1}$ is obtained by
\begin{itemize}

\item removing the row $i$ and the column $c$ from $w_i$, then by

\item multiplying the first column of the new matrix by $w_i[i,c]$.

\end{itemize}

Moreover, if we denote by $a_i$ the unique nonzero entry on the row $i$ of $w$, we have $w_1 = \displaystyle\prod_{i=1}^{n} a_i = \zeta_d^k$ for $0 \leq k \leq d-1$.

\end{lemma}

\begin{example}

Let $w$ be as in Example \ref{ExampleAlgo1}, where $n=4$. The block $w_3$ is obtained by removing the row number 4 and the second column from $w_4 = w$, to obtain $\begin{pmatrix}

\zeta_9 & 0 & 0\\
0 & 1 & 0\\
0 & 0 & \zeta_9\\

\end{pmatrix}$, then by multiplying the first column of this matrix by $\zeta_9$. The same can be said for the other block $w_2$. Finally, the block $w_1$ is equal to $\zeta_3$ which corresponds to the product of the nonzero entries of $w$.

\end{example}

\begin{definition}\label{DefREiwG(de,e,n)}

Let $1 \leq i \leq n$. Let $w_i[i,c] \neq 0$ for $1 \leq c \leq i$.

\begin{itemize}

\item If $w_1 = \zeta_d^k$ with $0 \leq k \leq d-1$, we define $R\!E_{1}(w)$ to be the word $\mathbf{z}^{k}$.	
\item If $w_{i}[i,c] =1$, we define $R\!E_{i}(w)$ to be the word\\
$\mathbf{s}_i \mathbf{s}_{i-1} \cdots \mathbf{s}_{c+1}$ (decreasing-index expression).
\item If $w_{i}[i,c] = \zeta_{de}^{k}$ with $k \neq 0$, we define $R\!E_{i}(w)$ to be the word\\
\begin{tabular}{ll}
			$\mathbf{s}_i \cdots \mathbf{s}_3 \mathbf{t}_k$ & if $c=1$,\\
			$\mathbf{s}_i \cdots \mathbf{s}_3 \mathbf{t}_k \mathbf{t}_0$  & if $c=2$,\\
			$\mathbf{s}_i \cdots \mathbf{s}_3 \mathbf{t}_k \mathbf{t}_0 \mathbf{s}_3 \cdots \mathbf{s}_c$ & if $c \geq 3$.\\
			
\end{tabular}

\end{itemize}
\end{definition}

The output $R\!E(w)$ of Algorithm \ref{Algo2} is a concatenation of the words $R\!E_1(w)$, $R\!E_2(w)$, $\cdots$, and $R\!E_n(w)$ obtained at each step $i$ from $n$ to $1$ of Algorithm \ref{Algo2}. Then, we have $R\!E(w) = R\!E_1(w) R\!E_2(w)  \cdots R\!E_n(w)$.

\begin{example}

Let $w$ be defined as in Example \ref{ExampleAlgo1}. We have
$$R\!E(w) = \underset{R\!E_{1}(w)}{\underbrace{\mathbf{z}}} \hspace{0.2cm} \underset{R\!E_{3}(w)}{\underbrace{\mathbf{s}_3 \mathbf{t}_1 \mathbf{t}_0 \mathbf{s}_3}} \hspace{0.2cm} \underset{R\!E_{4}(w)}{\underbrace{ \mathbf{s}_4 \mathbf{s}_3 \mathbf{t}_1 \mathbf{t}_0}}.$$ In this example, $R\!E_2(w)$ is the empty word.

\end{example}

\begin{proposition}\label{PropWordRepG(de,e,n)}

Let $w \in G(de,e,n)$. The word $R\!E(w)$ given by Algorithm \ref{Algo2} is a word representative over $\mathbf{X}$ of $w \in G(de,e,n)$.

\end{proposition}

\begin{proof}

Let $w \in G(de,e,n)$ such that the product of all the nonzero entries of $w$ is equal to $\zeta_d^k$ for some $0 \leq k \leq d-1$.  Algorithm \ref{Algo2} transforms the matrix $w$ into $I_n$ by multiplying it on the right by elements of $X$. We get $w x_1 \cdots x_{r-1}x_r = I_n$, where $x_1, \cdots, x_{r-1}$ are elements of $X \setminus \{z\}$ and $x_r = z^{-k}$. Hence $w = x_r^{-1} x_{r-1}^{-1} \cdots x_1^{-1} = z^{k} x_{r-1} \cdots x_1$ since $x_i^2 = 1$ for $1 \leq i \leq r-1$. The output $R\!E(w)$ of Algorithm \ref{Algo2} is $R\!E(w) = \mathbf{z}^k \mathbf{x}_{r-1} \cdots \mathbf{x}_1$. Hence it is a word representative over $\mathbf{X}$ of $w \in G(de,e,n)$.  

\end{proof}

The following proposition will prepare us to prove that the output of \mbox{Algorithm \ref{Algo2}} is a reduced expression over $\mathbf{X}$ of a given element $w \in G(de,e,n)$.

\begin{proposition}\label{prop.lengthcompG(de,e,n)}

Let $w$ be an element of $G(de,e,n)$. For all $x \in X$, we have $$\boldsymbol{\ell}(R\!E(xw)) \leq \boldsymbol{\ell}(R\!E(w)) + 1.$$

\end{proposition}

\begin{proof}

For $1 \leq i \leq n$, there exists a unique $c_i$ such that $w[i,c_i] \neq 0$. We denote $w[i,c_i]$ by $a_i$. We have  $\displaystyle\prod_{i=1}^{n} a_i = \zeta_d^k$ for some $0 \leq k \leq d-1$. \\

\underline{Case 1: Suppose $x = s_i$ for $3 \leq i \leq n$.}\\

A similar case is done in our previous work \cite{GeorgesNeaimeIntervals}. We get $(a)$, $(b)$, $(a')$ and $(b')$ as in the proof of Proposition 3.11 (Case 1) in \cite{GeorgesNeaimeIntervals}. For completeness, we include this part of the proof in this article.\\

Set $w' := s_i w$. Since the left multiplication by the matrix $x$ exchanges the rows $i-1$ and $i$ of $w$ and the other rows remain the same, by Definition \ref{DefREiwG(de,e,n)} and Lemma \ref{LemmaBlocksG(de,e,n)}, we have:\\
$R\!E_{i+1}(xw)R\!E_{i+2}(xw) \cdots R\!E_{n}(xw)=R\!E_{i+1}(w) R\!E_{i+2}(w) \cdots R\!E_{n}(w)$ and\\
$R\!E_{2}(xw)R\!E_{3}(xw) \cdots R\!E_{i-2}(xw)=R\!E_{2}(w) R\!E_{3}(w) \cdots R\!E_{i-2}(w)$.\\
Then, in order to prove our property, we should compare $\boldsymbol{\ell}_1:=\boldsymbol{\ell}(R\!E_{i-1}(w)R\!E_i(w))$ and $\boldsymbol{\ell}_2:=\boldsymbol{\ell}(R\!E_{i-1}(xw)R\!E_i(xw))$.\\

Suppose  $c_{i-1} < c_{i}$, by Lemma \ref{LemmaBlocksG(de,e,n)}, the rows $i-1$ and $i$ of the blocks $w_i$ and $w'_{i}$ are of the form:

\begin{tabular}{ccc}

\begin{tikzpicture}[scale=0.5]

\node at (0,0) {};
\node at (0,0.8) {$w_i$ :};

\end{tikzpicture}

 & & \begin{tikzpicture}[scale=0.5]

\node at (0,1) {$i$};
\node at (0,2) {$i-1$};
\node at (2,3) {$..$};
\node at (2.6,3) {$c$};
\node at (4,3) {$..$};
\node at (5,3) {$c'$};
\node at (5.7,3) {$..$};
\node at (7,3) {$i$};

\node at (3,2) {$b_{i-1}$};
\node at (5,1) {$a_i$};

\draw [-] (1,0.5) to node[auto] {} (1,2.5);
\draw [-] (1,0.5) to node[auto] {} (7.3,0.5);
\draw [-] (7.3,0.5) to node[auto] {} (7.3,2.5);
\draw [-] (1,2.5) to node[auto] {} (7.3,2.5);

\end{tikzpicture}\\

\begin{tikzpicture}[scale=0.5]

\node at (0,0) {};
\node at (0,0.8) {$w'_i$ :};

\end{tikzpicture}

 & & \begin{tikzpicture}[scale=0.5]

\node at (0,1) {$i$};
\node at (0,2) {$i-1$};
\node at (2,3) {$..$};
\node at (2.6,3) {$c$};
\node at (4,3) {$..$};
\node at (5,3) {$c'$};
\node at (5.7,3) {$..$};
\node at (7,3) {$i$};

\node at (3,1) {$b_{i-1}$};
\node at (5,2) {$a_i$};

\draw [-] (1,0.5) to node[auto] {} (1,2.5);
\draw [-] (1,0.5) to node[auto] {} (7.3,0.5);
\draw [-] (7.3,0.5) to node[auto] {} (7.3,2.5);
\draw [-] (1,2.5) to node[auto] {} (7.3,2.5);

\end{tikzpicture}

\end{tabular}

with $c < c'$ and where we write $b_{i-1}$ instead of $a_{i-1}$ since $a_{i-1}$ may change when applying Algorithm \ref{Algo2} if $c_{i-1} =1$, that is $a_{i-1}$ on the first column of $w$ (see the second item of Lemma \ref{LemmaBlocksG(de,e,n)}).\\

We will discuss different cases depending on the values of $a_i$ and $b_{i-1}$.

\begin{itemize}

\item \underline{Suppose $a_i = 1$.}

\begin{itemize}

\item \underline{If $b_{i-1} =1$,}\\
we have $R\!E_{i}(w) = \boldsymbol{s}_i \cdots \boldsymbol{s}_{c'+2} \boldsymbol{s}_{c'+1}$ and $R\!E_{i-1}(w) = \boldsymbol{s}_{i-1} \cdots \boldsymbol{s}_{c+2} \boldsymbol{s}_{c+1}$.
Furthermore, we have $R\!E_{i}(xw) = \boldsymbol{s}_i \cdots \boldsymbol{s}_{c+2} \boldsymbol{s}_{c+1}$\\
and $R\!E_{i-1}(xw) = \boldsymbol{s}_{i-1} \cdots \boldsymbol{s}_{c'+1} \boldsymbol{s}_{c'}$.\\
It follows that $\boldsymbol{\ell}_1 = ((i-1)-(c+1)+1) + (i-(c'+1)+1) = 2i-c-c'-1$ and $\boldsymbol{\ell}_2 = ((i-1)-c'+1) + (i-(c+1)+1) = 2i-c-c'$ hence $\boldsymbol{\ell}_2 = \boldsymbol{\ell}_1 +1$.\\
\item \underline{If $b_{i-1} = \zeta_{de}^{k}$ with $1 \leq k \leq de-1$,}\\
we have $R\!E_{i}(w) = \boldsymbol{s}_i \cdots \boldsymbol{s}_{c'+2} \boldsymbol{s}_{c'+1}$ and $R\!E_{i-1}(w) = \boldsymbol{s}_{i-1} \cdots \boldsymbol{s}_3 \boldsymbol{t}_k \boldsymbol{t}_0 \boldsymbol{s}_{3} \cdots \boldsymbol{s}_{c}$. Furthermore, we have $R\!E_{i}(xw) = \boldsymbol{s}_i \cdots \boldsymbol{s}_3 \boldsymbol{t}_k \boldsymbol{t}_0 \boldsymbol{s}_{3} \cdots \boldsymbol{s}_{c}$ and $R\!E_{i-1}(xw) = \boldsymbol{s}_{i-1} \cdots \boldsymbol{s}_{c'}$.\\
It follows that $\boldsymbol{\ell}_1 = (((i-1)-3+1) + 2 + (c-3+1)) + (i-(c'+1)+1) = 2i+c-c'-3$ and $\boldsymbol{\ell}_2 = ((i-1)-c'+1) + ((i-3+1) + 2 + (c-3+1)) = 2i+c-c'-2$ hence $\boldsymbol{\ell}_2 = \boldsymbol{\ell}_1 +1$.

\end{itemize}

It follows that

\begin{center}
if $a_i =1$, then $\boldsymbol{\ell}(R\!E(s_iw))= \boldsymbol{\ell}(R\!E(w)) +1. \hspace{1cm} (a)$
\end{center}

\item \underline{Suppose now that $a_i = \zeta_{de}^{k}$ with $1 \leq k \leq de-1$.}

\begin{itemize}

\item \underline{If $b_{i-1} = 1$,}\\
we have $R\!E_{i}(w) = \boldsymbol{s}_{i} \cdots \boldsymbol{s}_3 \boldsymbol{t}_k \boldsymbol{t}_0 \boldsymbol{s}_{3} \cdots \boldsymbol{s}_{c'}$ and $R\!E_{i-1}(w) =\boldsymbol{s}_{i-1} \cdots \boldsymbol{s}_{c+1}$.\\
Also, we have $R\!E_{i}(xw) = \boldsymbol{s}_i \cdots \boldsymbol{s}_{c+1}$ and\\
$R\!E_{i-1}(xw) = \boldsymbol{s}_{i-1} \cdots \boldsymbol{s}_3 \boldsymbol{t}_k \boldsymbol{t}_0 \boldsymbol{s}_{3} \cdots \boldsymbol{s}_{c'-1}$.\\
It follows that $\boldsymbol{\ell}_1 = ((i-1)-(c+1)-1) + ((i-3+1)+2+(c'-3+1)) = 2i-c+c'-5$ and $\boldsymbol{\ell}_2 = (((i-1)-3+1)+2+((c'-1)-3+1)) + (i-(c+1)-1) = 2i-c+c'-6$ hence $\boldsymbol{\ell}_2 = \boldsymbol{\ell}_1 -1$.\\

\item \underline{If $b_{i-1} = \zeta_{de}^{k'}$ with $1 \leq k' \leq de-1$,}\\
we have $R\!E_{i}(w) = \boldsymbol{s}_{i} \cdots \boldsymbol{s}_3 \boldsymbol{t}_k \boldsymbol{t}_0 \boldsymbol{s}_{3} \cdots  \boldsymbol{s}_{c'}$ and\\
$R\!E_{i-1}(w) = \boldsymbol{s}_{i-1} \cdots \boldsymbol{s}_3 \boldsymbol{t}_{k'} \boldsymbol{t}_0 \boldsymbol{s}_{3} \cdots \boldsymbol{s}_{c}$. \\
Also, we have $R\!E_{i}(xw) = \boldsymbol{s}_{i} \cdots \boldsymbol{s}_3 \boldsymbol{t}_{k'} \boldsymbol{t}_0 \boldsymbol{s}_{3} \cdots \boldsymbol{s}_{c}$ and\\
$R\!E_{i-1}(xw) = \boldsymbol{s}_{i-1} \cdots \boldsymbol{s}_3 \boldsymbol{t}_k \boldsymbol{t}_0 \boldsymbol{s}_{3} \cdots \boldsymbol{s}_{c'-1}$.\\
It follows that $\boldsymbol{\ell}_1 = ((i-1)-3+1) +2+(c-3+1) + (i-3+1) +2+ (c'-3+1) = 2i+c+c'-5$ and $\boldsymbol{\ell}_2 = ((i-1)-3+1)+2+((c'-1)-3+1)+(i-3+1)+2+(c-3+1) = 2i+c+c'-6$ hence $\boldsymbol{\ell}_2 = \boldsymbol{\ell}_1 -1$.

\end{itemize}

It follows that

\begin{center}
if $a_i \neq 1$, then $\boldsymbol{\ell}(R\!E(s_iw)) = \boldsymbol{\ell}(R\!E(w)) -1. \hspace{1cm} (b)$
\end{center}

\end{itemize}

Suppose, on the other hand, $c_{i-1} > c_i$. Recall that $w' = s_iw$. If $w'[i-1,c'_{i-1}]$ and $w'[i,c'_{i}]$ denote the nonzero entries of $w'$ on the rows $i-1$ and $i$, respectively, we have $w'[i-1,c'_{i-1}] = a_i$ and $w'[i,c'_{i}] = a_{i-1}$. For $w'$, we have $c'_{i-1} < c'_{i}$, in which case the preceding analysis would give:

\begin{center}

if $a_{i-1} = 1$, then $\boldsymbol{\ell}(R\!E(s_i(s_iw))) = \boldsymbol{\ell}(R\!E(s_iw)) + 1$,\\
if $a_{i-1} \neq 1$, then $\boldsymbol{\ell}(R\!E(s_i(s_iw))) = \boldsymbol{\ell}(R\!E(s_iw)) - 1$.

\end{center}

Hence, since $s_i^2 = 1$, we get the following:

\begin{center}

if $a_{i-1} = 1$, then $\boldsymbol{\ell}(R\!E(s_iw)) = \boldsymbol{\ell}(R\!E(w)) - 1. \hspace{1cm} (a')$,\\
if $a_{i-1} \neq 1$, then $\boldsymbol{\ell}(R\!E(s_iw)) = \boldsymbol{\ell}(R\!E(w)) + 1. \hspace{1cm} (b')$.

\end{center}

\underline{Case 2: Suppose $x = t_i$ for $0 \leq i \leq de-1$.}\\

Set $w' := t_iw$. By the left multiplication by $t_i$, we have that the last $n-2$ rows of $w$ and $w'$ are the same. Hence, by Definition \ref{DefREiwG(de,e,n)} and Lemma \ref{LemmaBlocksG(de,e,n)}, we have:\\
$R\!E_3(xw)R\!E_4(xw) \cdots R\!E_n(xw) = R\!E_3(w)R\!E_4(w) \cdots R\!E_n(w)$. In order to prove our property in this case, we should compare $\boldsymbol{\ell}_1 := \boldsymbol{\ell}(R\!E_1(w)R\!E_2(w))$ and\\ $\boldsymbol{\ell}_2 := \boldsymbol{\ell}(R\!E_1(xw)R\!E_2(xw))$.

\begin{itemize}

\item \underline{Consider the case where $c_1 < c_2$.}\\

Since $c_1 < c_2$, by Lemma \ref{LemmaBlocksG(de,e,n)}, the blocks $w_2$ and $w'_2$ are of the form:

$w_2 = \begin{pmatrix}

b_1 & 0\\
0 & a_2\\

\end{pmatrix}$ and $w'_2 = \begin{pmatrix}

0 & \zeta_{de}^{-i}a_2\\
\zeta_{de}^{i}b_{1} & 0\\ 

\end{pmatrix}$ with $b_{1}$ instead of $a_{1}$ since $a_{1}$ may change when applying Algorithm \ref{Algo2} if $c_{1} =1$.

\begin{itemize}
\item \underline{Suppose $a_2 = 1$,}\\
we have $b_1 = \zeta_d^k$ hence $\boldsymbol{\ell}_1 = k$. We also have $R\!E_2(xw) = \boldsymbol{t}_{i+ke}$ and $R\!E_1(xw) = \mathbf{z}^k$. Hence we get $\boldsymbol{\ell}_2 = k+1$. It follows that when $c_1 < c_2$,

\begin{center}
if $a_2 =1$, then $\boldsymbol{\ell}(R\!E(t_iw)) = \boldsymbol{\ell}(R\!E(w)) +1. \hspace{1cm} (c)$
\end{center}

\item \underline{Suppose $a_2 = \zeta_{de}^{k'}$ with $1 \leq k' \leq de-1$,}\\
we have $b_1 = \zeta_{de}^{ke-k'}$. We get $R\!E_2(w) = \boldsymbol{t}_{k'}\boldsymbol{t}_0$ and $R\!E_1(w) = \mathbf{z}^k$. Thus, $\boldsymbol{\ell}_1 = k+2$. We also get $R\!E_2(xw) = \boldsymbol{t}_{ke+i-k'}$ and $R\!E_1(xw) = \mathbf{z}^k$. Thus, $\boldsymbol{\ell}_2 = k+1$. It follows that when $c_1 < c_2$,

\begin{center}
if $a_2 \neq 1$, then $\boldsymbol{\ell}(R\!E(t_iw)) = \boldsymbol{\ell}(R\!E(w))  -1. \hspace{1cm} (d)$
\end{center}

\end{itemize}

\item \underline{Now, consider the case where $c_1 > c_2$.}\\

Since $c_1 > c_2$, by Lemma \ref{LemmaBlocksG(de,e,n)}, the blocks $w_2$ and $w'_2$ are of the form:

$w_2 = \begin{pmatrix}

0 & a_1\\
b_2 & 0\\

\end{pmatrix}$ and $w'_2 = \begin{pmatrix}

\zeta_{de}^{-i}b_2 & 0\\
0 & \zeta_{de}^{i}a_{1}\\ 

\end{pmatrix}$ with $b_{2}$ instead of $a_{2}$ since $a_{2}$ may change when applying Algorithm \ref{Algo2} if $c_{2} =1$.

\begin{itemize}

\item \underline{Suppose $a_1 \neq \zeta_{de}^{-i}$,}\\
we have $\boldsymbol{\ell}_1 = k+1$, and since $\zeta_{de}^{i}a_1 \neq 1$, we have $\boldsymbol{\ell}_2 = k+2$. Hence when $c_1 > c_2$,

\begin{center}
if $a_1 \neq \zeta_{de}^{-i}$, then $\boldsymbol{\ell}(R\!E(t_iw)) = \boldsymbol{\ell}(R\!E(w)) +1. \hspace{1cm} (e)$
\end{center}

\item \underline{Suppose $a_1 = \zeta_{de}^{-i}$,} we have $b_2 = \zeta_{de}^{i+ek}$.\\
We get $\boldsymbol{\ell}_1 = k+1$ and $\boldsymbol{\ell}_2 = k$. Hence when $c_1 > c_2$,

\begin{center}
if $a_1 = \zeta_{de}^{-i}$, then $\boldsymbol{\ell}(R\!E(t_iw)) = \boldsymbol{\ell}(R\!E(w)) -1. \hspace{1cm} (f)$
\end{center}

\end{itemize}

\end{itemize}

\underline{Case 3: Suppose $x = z$.}\\

Set $w' := zw$. By the left multiplication by $z$, we have that the last $n-1$ rows of $w$ and $w'$ are the same. Hence, by Definition \ref{DefREiwG(de,e,n)} and Lemma \ref{LemmaBlocksG(de,e,n)}, we have:\\
$R\!E_2(xw)R\!E_3(xw) \cdots R\!E_n(xw) = R\!E_2(w)R\!E_3(w) \cdots R\!E_n(w)$. In order to prove our property in this case, we should compare $\boldsymbol{\ell}_1 := \boldsymbol{\ell}(R\!E_1(w))$ and $\boldsymbol{\ell}_2 := \boldsymbol{\ell}(R\!E_1(xw))$.

We get $w_1$ is equal to $b_1$ and $w'_1 = \zeta_d b_1$ with $b_{1}$ instead of $a_{1}$ since $a_{1}$ may change when applying Algorithm \ref{Algo2} if $c_{1} =1$. We have $b_1 = \displaystyle\prod_{i=1}^{n} a_i = \zeta_d^k$ for some $0 \leq k \leq d-1$. Hence if $k \neq d-1$, we get $\boldsymbol{\ell}_1 = k$ and $\boldsymbol{\ell}_2 = k+1$ and if $k = d-1$, we get $\boldsymbol{\ell}_1 = d-1$ and $\boldsymbol{\ell}_2 = 0$. It follows that

\begin{center}
$\boldsymbol{\ell}(R\!E(zw)) \leq \boldsymbol{\ell}(R\!E(w)) + 1. \hspace{1cm} (g)$
\end{center}

\end{proof}

The next proposition establishes that Algorithm \ref{Algo2} produces geodesic normal forms for $G(de,e,n)$.

\begin{proposition}\label{PropREwRedExp}

Let $w$ be an element of $G(de,e,n)$. The word $R\!E(w)$ is a reduced expression over $\mathbf{X}$ of $w$.

\end{proposition}

\begin{proof}

We must prove that $\ell(w) = \boldsymbol{\ell}(R\!E(w))$. Let $\mathbf{x}_1\mathbf{x}_2 \cdots \mathbf{x}_r$ be a reduced expression over $\mathbf{X}$ of $w$. Hence $\ell(w) = \boldsymbol{\ell}(\mathbf{x}_1\mathbf{x}_2 \cdots \mathbf{x}_r) = r$. Since $R\!E(w)$ is a word representative over $\mathbf{X}$ of $w$, we have $\boldsymbol{\ell}(R\!E(w)) \geq \boldsymbol{\ell}(\mathbf{x}_1\mathbf{x}_2 \cdots \mathbf{x}_r) = r$. 

We prove that $\boldsymbol{\ell}(R\!E(w)) \leq r$. Write $w$ as $x_1x_2 \cdots x_r$ where $x_1,x_2, \cdots, x_r$ are the matrices of $G(e,e,n)$ corresponding to $\mathbf{x}_1,\mathbf{x}_2, \cdots, \mathbf{x}_r$. By Proposition \ref{prop.lengthcompG(de,e,n)}, we have: $\boldsymbol{\ell}(R\!E(w)) = \boldsymbol{\ell}(R\!E(x_1x_2 \cdots x_r)) \leq \boldsymbol{\ell}(R\!E(x_2x_3 \cdots x_r)) +1 \leq \boldsymbol{\ell}(R\!E(x_3 \cdots x_r)) +2 \leq \cdots \leq r$. Hence $\boldsymbol{\ell}(R\!E(w))$ is equal to $\ell(w)$. This establishes that $R\!E(w)$ is a reduced expression over $\mathbf{X}$ of $w$.

\end{proof}

\begin{remark}\label{GeodesicFormsForG(e,e,n)}

Geodesic normal forms for the complex reflection groups $G(e,e,n)$ have been already established in our previous work \cite{GeorgesNeaimeIntervals}. They are explicitly defined by an algorithm (similar to Algorithm \ref{Algo2}). Let $w \in G(e,e,n)$. The output of the algorithm is the word also denoted by $R\!E(w)$ and defined as a concatenation of the words $R\!E_2(w)$, $R\!E_3(w)$, $\cdots$, $R\!E_n(w)$ introduced as in Definition \ref{DefREiwG(de,e,n)}. It describes a minimal word representative of the element $w \in G(e,e,n)$.

\end{remark}

As a direct consequence of Algorithm \ref{Algo2} and Proposition \ref{PropREwRedExp}, the next statement characterizes the elements of $G(de,e,n)$ that are of maximal length over the generating set of Corran-Lee-Lee.

\begin{proposition}\label{PropMaxLengthG(de,e,n)}

Let $d > 1, e > 1$ and $n \geq 2$. The maximal length of an element of $G(de,e,n)$ over the generating set of Corran-Lee-Lee is \mbox{$n(n-1) + d-1$.} It is realized for diagonal matrices $w$ such that for all $2 \leq i \leq n$, we have $w[i,i] = \zeta_{de}^{k_i}$ with $1 \leq k_i \leq de-1$ and $w[1,1] = \zeta_{de}^x$ with $x + (k_2 \cdots k_n) = e(d-1)$. A minimal word representative of such an element is of the form $$\mathbf{z}^{d-1} (\mathbf{t}_{k_2}\mathbf{t}_0) (\mathbf{s}_3\mathbf{t}_{k_3}\mathbf{t}_0\mathbf{s}_3)\cdots (\mathbf{s}_n \cdots \mathbf{s}_3\mathbf{t}_{k_n}\mathbf{t}_0\mathbf{s}_3 \cdots \mathbf{s}_n),$$ with $1 \leq k_2, \cdots, k_n \leq de-1$. The number of elements that are of maximal length is then $(de-1)^{(n-1)}$.

\end{proposition}

\subsection{The case of $G(d,1,n)$}\label{SubGeodesicG(d,1,n)}

We establish a similar construction for the case of $G(d,1,n)$ for $d > 1$ and $n \geq 2$. We recall the diagram of the presentation of $G(d,1,n)$:

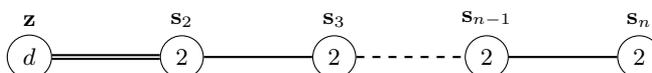
\begin{figure}[H]

\begin{small}
\begin{center}
\begin{tikzpicture}

\node[draw, shape=circle, label=above:$\mathbf{z}$] (1) at (0,0) {$d$};
\node[draw, shape=circle,label=above:$\mathbf{s}_2$] (2) at (2,0) {$2$};
\node[draw, shape=circle,label=above:$\mathbf{s}_3$] (3) at (4,0) {$2$};
\node[draw, shape=circle,label=above:$\mathbf{s}_{n-1}$] (n-1) at (6,0) {$2$};
\node[draw,shape=circle,label=above:$\mathbf{s}_n$] (n) at (8,0) {$2$};

\draw[thick,-,double] (1) to (2);
\draw[thick,-] (2) to (3);
\draw[dashed,-,thick] (3) to (n-1);
\draw[thick,-] (n-1) to (n);

\end{tikzpicture}
\end{center}
\end{small}

\end{figure}

Denote by $\mathbf{X}$ the set $\{\mathbf{z}, \mathbf{s}_2, \cdots, \mathbf{s}_n\}$ of the generators. The generator $\mathbf{z}$ corresponds to the matrix $z := Diag(\zeta_d,1,\cdots,1)$ in $G(d,1,n)$ with $\zeta_d = exp(2i\pi / d)$ and $\mathbf{s}_j$ corresponds to the transposition matrix $s_j := (j-1,j)$ for $2 \leq j \leq n$. Denote by $X$ the set $\{z,s_2,s_3, \cdots, s_n\}$ of these matrices.\\

We define Algorithm \ref{Algo3} below that produces a word $R\!E(w)$ for each matrix $w$ of $G(d,1,n)$. This Algorithm is different than Algorithm \ref{Algo2}. Let us explain the steps of the algorithm. Let $w_n := w \in G(d,1,n)$. For $i$ from $n$ to $1$, the $i$-th step of the algorithm transforms the block diagonal matrix $\left(
\begin{array}{c|c}
w_i & 0 \\
\hline
0 & I_{n-i}
\end{array}
\right)$ into a block diagonal matrix $\left(
\begin{array}{c|c}
w_{i-1} & 0 \\
\hline
0 & I_{n-i+1}
\end{array}
\right) \in G(d,1,n)$. Let $w_i[i,c] \neq 0$ be the nonzero coefficient on the row $i$ of $w_i$. If $w_i[i,c] =1$, we shift it into the diagonal position $[i,i]$ by right multiplication by transpositions. If $w_i[i,c] = \zeta_d^k$ with $k \geq 1$, we shift it into position $[i,1]$ by right multiplication by transpositions, followed by a right multiplication by $z^{-k}$, then we shift the $1$ obtained in position $[i,1]$ into the diagonal position $[i,i]$ by right multiplication by transpositions. Let us illustrate these operations by the following example.

\begin{algorithm}\label{Algo3}

\SetKwInOut{Input}{Input}\SetKwInOut{Output}{Output}

\noindent\rule{12cm}{0.5pt}

\Input{$w$, a matrix in $G(d,1,n)$, with $d > 1$ and $n \geq 2$.}
\Output{$R\!E(w)$, a word over $\mathbf{X}$.}

\noindent\rule{12cm}{0.5pt}

\textbf{Local variables}: $w'$, $R\!E(w)$,  $i$, $U$, $c$, $k$.

\noindent\rule{12cm}{0.5pt}

\textbf{Initialisation}:
$U:=[1,\zeta_d,\zeta_{d}^2, \cdots ,\zeta_{d}^{d-1}]$, $R\!E(w) := \varepsilon$: the empty word, $w' := w$.

\noindent\rule{12cm}{0.5pt}

\For{$i$ \textbf{from} $n$ \textbf{down to} $1$} {
	$c:=1$; $k:=0$; \\
	\While{$w'[i,c] = 0$}{$c :=c+1$\; 
	}
	 \textit{\#Then $w'[i,c]$ is the root of unity on the row $i$}\;
	\While{$U[k+1]\neq w'[i,c] $}{$k :=k+1$\;
	}
	\textit{\#Then $w'[i,c] = \zeta_{d}^k$.}\\

		\If{$k \neq 0$}{
		$w' := w's_{c}s_{c-1} \cdots s_{3}s_{2}z^{-k}$; \textit{\#Then $w'[i,2] =1$}\;
		$R\!E(w) := \mathbf{z}^{k}\mathbf{s}_2\mathbf{s}_3 \cdots \mathbf{s}_c R\!E(w)$\;
		$c:=1$\;
	}	
	$w' := w's_{c+1} \cdots s_{i-1} s_{i}$; \textit{\#Then $w'[i,i] = 1$}\;
	$R\!E(w) := \mathbf{s}_i \mathbf{s}_{i-1} \cdots \mathbf{s}_{c+1} R\!E(w)$\;
}
\textbf{Return} $R\!E(w)$;

\noindent\rule{12cm}{0.5pt}

\caption{A word over $\mathbf{X}$ corresponding to an element $w \in G(d,1,n)$.}
\end{algorithm}

\begin{example}\label{ExampAlgoG(d,1,n)}

Let $w := \begin{pmatrix}

0 & \zeta_3 & 0\\
0 & 0 & \zeta_3^2\\
\zeta_3^2 & 0 & 0

\end{pmatrix}$ $\in G(3,1,3)$.\\
Step $1$ $(i=3, k=2, c=1)$: $w':= w z^{-2}=\begin{pmatrix}

0 & \zeta_3 & 0\\
0 & 0 & \zeta_3^2\\
1 & 0 & 0

\end{pmatrix}$, then\\ $w':= w' s_2 s_3 = \begin{pmatrix}

\zeta_3 & 0 & 0\\
0 & \boxed{\zeta_3^2} & 0\\
0 & 0 & \mathbf{1}

\end{pmatrix}$.\\
Step $2$ $(i= 2, k = 2, c=1)$: $w' := w' s_2 = \begin{pmatrix}

0 & \zeta_3 & 0\\
\zeta_3^2 & 0 & 0\\
0 & 0 & 1

\end{pmatrix}$, then $w':= w' z^{-2} = \begin{pmatrix}

0 & \zeta_3 & 0\\
1 & 0 & 0\\
0 & 0 & 1

\end{pmatrix}$, then $w' := w' s_2 = \begin{pmatrix}

\boxed{\zeta_3} & 0 & 0\\
0 & \mathbf{1} & 0\\
0 & 0 & 1

\end{pmatrix}$.\\
Step $3$ $(i=1, k=1, c=1)$: $w' := w' z^{-1} = I_3$.\\
Hence $R\!E(w) = \mathbf{z} \mathbf{s}_2 \mathbf{z}^2 \mathbf{s}_2 \mathbf{s}_3 \mathbf{s}_2 \mathbf{z}^2$.

\end{example}

The next lemma follows directly from Algorithm \ref{Algo3}.

\begin{lemma}\label{LemmaBlocksG(d,1,n)}

For $2 \leq i \leq n$, let $w_i[i,c] \neq 0$ be the nonzero coefficient on the row $i$ of $w_i$. The block $w_{i-1}$ is obtained by removing the row $i$ and the column $c$ from $w_i$. Moreover, $w_1$ is equal to the nonzero entry on the first row of $w$.

\end{lemma}

\begin{definition}\label{DefinitionREG(d,1,n)}

Let $1 \leq i \leq n$. Let $w_i[i,c] \neq 0$ for $1 \leq c \leq i$.

\begin{itemize}

\item If $w_1 = \zeta_d^k$ for some $0 \leq k \leq d-1$ (this is equal to the nonzero entry on the first row of $w$), we define $R\!E_{1}(w)$ to be the word $\mathbf{z}^{k}$.	
\item If $w_{i}[i,c] =1$, we define $R\!E_{i}(w)$ to be the word\\
$\mathbf{s}_i \mathbf{s}_{i-1} \cdots \mathbf{s}_{c+1}$ (decreasing-index expression).
\item If $w_{i}[i,c] = \zeta_{d}^{k}$ with $k \neq 0$, we define $R\!E_{i}(w)$ to be the word\\
\begin{tabular}{ll}
			$\mathbf{s}_i \cdots \mathbf{s}_3 \mathbf{z}^k$ & if $c=1$,\\
			$\mathbf{s}_i \cdots \mathbf{s}_3 \mathbf{s}_2 \mathbf{z}^k \mathbf{s}_2 \mathbf{s}_3 \cdots \mathbf{s}_c$ & if $c \geq 2$.\\
			
\end{tabular}

\end{itemize}

\end{definition}

As for Algorithm \ref{Algo2}, the output of Algorithm \ref{Algo3} is equal to $R\!E_1(w) R\!E_2(w) \cdots R\!E_n(w)$. In Example \ref{ExampAlgoG(d,1,n)}, we have $R\!E(w) = \underset{R\!E_1(w)}{\underbrace{\mathbf{z}}} \hspace{0.2cm} \underset{R\!E_2(w)}{\underbrace{\mathbf{s}_2 \mathbf{z}^2 \mathbf{s}_2}} \hspace{0.2cm} \underset{R\!E_3(w)}{\underbrace{\mathbf{s}_3 \mathbf{s}_2 \mathbf{z}^2}}$.\\

The proof of the next proposition is similar to the proof of Proposition \ref{PropWordRepG(de,e,n)} and is left to the reader.

\begin{proposition}\label{PropWordRepG(d,1,n)}

Let $w \in G(d,1,n)$. The word $R\!E(w)$ given by Algorithm \ref{Algo3} is a word representative over $\mathbf{X}$ of $w \in G(d,1,n)$.

\end{proposition}

The following proposition enables us to prove that the output of Algorithm \ref{Algo3} is a reduced expression over $\mathbf{X}$ of a given element $w \in G(d,1,n)$.

\begin{proposition}\label{prop.lengthcompG(d,1,n)}

Let $w$ be an element of $G(d,1,n)$. For all $x \in X$, we have $$\boldsymbol{\ell}(R\!E(xw)) \leq \boldsymbol{\ell}(R\!E(w)) + 1.$$

\end{proposition}

\begin{proof}

Consider $x = s_i$ with $2 \leq i \leq n$. This case is done in the same way as Case 1 in the proof of Proposition \ref{prop.lengthcompG(de,e,n)}. Consider now $x = z$. This case is done this time as Case 3 of the proof of Proposition \ref{prop.lengthcompG(de,e,n)}.

\end{proof}

Applying the arguments used before in the proof of Proposition \ref{PropREwRedExp}, we deduce that $R\!E(w)$ is a reduced expression over $\mathbf{X}$ of $w \in G(d,1,n)$. Hence \mbox{Algorithm \ref{Algo3}} produces geodesic normal forms for $G(d,1,n)$.\\

As a direct application, the next statement characterizes the elements of $G(d,1,n)$ that are of maximal length. Note that this statement was also observed in \cite{BremkeMalle}.

\begin{proposition}\label{PropMaxLengthG(d,1,n)}

Let $d > 1$ and $n \geq 2$. There exists a unique element of maximal length of $G(d,1,n)$. Its minimal word representative is of the form $$\mathbf{z}^{d-1} (\mathbf{s}_2 \mathbf{z}^{d-1}\mathbf{s}_2) (\mathbf{s}_3\mathbf{s}_2\mathbf{z}^{d-1}\mathbf{s}_2\mathbf{s}_3) \cdots (\mathbf{s}_n \cdots \mathbf{s}_2 \mathbf{z}^{d-1} \mathbf{s}_2 \cdots \mathbf{s}_n).$$ Its length is then equal to $n(n+d-2)$.

\end{proposition}

\begin{remark}

When $d=2$, the group $G(2,1,n)$ is the Coxeter group of type $B_n$. By Proposition \ref{PropMaxLengthG(d,1,n)}, the longest element is of the form $$\mathbf{z} (\mathbf{s}_2 \mathbf{z}\mathbf{s}_2) (\mathbf{s}_3\mathbf{s}_2\mathbf{z}\mathbf{s}_2\mathbf{s}_3) \cdots (\mathbf{s}_n \cdots \mathbf{s}_2 \mathbf{z} \mathbf{s}_2 \cdots \mathbf{s}_n).$$ Its length is equal to $n^2$ which is already known for Coxeter groups of type $B_n$, see Example 1.4.6 of \cite{GeckPfeifferBook}.

\end{remark}

\section{The Hecke algebras \texorpdfstring{$H(de,e,n)$}{TEXT}}\label{SectionHeckeAlgebras}

The Hecke algebras $H(de,e,n)$ attached to the general series of complex reflection groups $G(de,e,n)$ are defined as quotients of the corresponding complex braid group algebras by some polynomial relations. Let $B(de,e,n)$ denotes the complex braid group attached to $G(de,e,n)$, as defined in \cite{BMR}. We establish nice presentations for the Hecke algebras $H(de,e,n)$ by using the presentations of Corran-Picantin \cite{CorranPicantin} and Corran-Lee-Lee \cite{CorranLeeLee} of the complex braid groups $B(e,e,n)$ and $B(de,e,n)$ for $d>1$, respectively.

\subsection{The Hecke algebras $H(e,e,n)$}

Corran-Picantin introduced in \cite{CorranPicantin} a presentation for the complex braid groups $B(e,e,n)$. Although we are
not going to use this result, we mention that they also established nice combinatorial structures (called Garside structures) for these groups. The presentation of Corran-Picantin is as follows.

\begin{definition}\label{DefPresB(e,e,n)CorranPicantin}

Let $e \geq 1$ and $n \geq 2$. The group $B(e,e,n)$ is defined by a presentation with set of generators: $\{t_i\ |\ i \in \mathbb{Z}/e\mathbb{Z}\}$ $\cup \{s_3, s_4, \cdots, s_n \}$ and relations:

\begin{enumerate}

\item $t_i t_{i-1} = t_j t_{j-1}$ for $i, j \in \mathbb{Z}/e\mathbb{Z}$,
\item $t_i s_3 t_i = s_3 t_i s_3$ for $i \in \mathbb{Z}/e\mathbb{Z}$,
\item $s_j t_i = t_i s_j$ for $i \in \mathbb{Z}/e\mathbb{Z}$ and $4 \leq j \leq n$,
\item $s_i s_{i+1} s_i = s_{i+1} s_i s_{i+1}$ for $3 \leq i \leq n-1$,
\item $s_i s_j = s_j s_i$ for $|i-j| > 1$.

\end{enumerate}

\end{definition}

Adding the quadratic relations to all the generators, we get the presentation of Corran-Picantin of $G(e,e,n)$ given earlier in Definition \ref{DefPresCorranPicantinG(e,e,n)}. Similar to the diagram of Figure \ref{FigureDiagramCPG(e,e,n)}, the diagram that describes the presentation of Corran-Picantin of $B(e,e,n)$ is the following. 

\begin{figure}[H]
\begin{center}
\begin{tikzpicture}[yscale=0.8,xscale=1,rotate=30]

\draw[thick,dashed] (0,0) ellipse (2cm and 1cm);

\node[draw, shape=circle, fill=white, label=above:\begin{small}$t_0$\end{small}] (t0) at (0,-1) {};
\node[draw, shape=circle, fill=white, label=above:\begin{small}$t_1$\end{small}] (t1) at (1,-0.8) {};
\node[draw, shape=circle, fill=white, label=right:\begin{small}$t_2$\end{small}] (t2) at (2,0) {};
\node[draw, shape=circle, fill=white, label=above:$t_i$] (ti) at (0,1) {};
\node[draw, shape=circle, fill=white, label=above:\begin{small}$t_{e-1}$\end{small}] (te-1) at (-1,-0.8) {};

\draw[thick,-] (0,-2) arc (-180:-90:3);

\node[draw, shape=circle, fill=white, label=below left:$s_3$] (s3) at (0,-2) {};

\draw[thick,-] (t0) to (s3);
\draw[thick,-,bend left] (t1) to (s3);
\draw[thick,-,bend left] (t2) to (s3);
\draw[thick,-,bend left] (s3) to (te-1);

\node[draw, shape=circle, fill=white, label=below:$s_4$] (s4) at (0.15,-3) {};
\node[draw, shape=circle, fill=white, label=below:$s_{n-1}$] (sn-1) at (2.2,-4.9) {};
\node[draw, shape=circle, fill=white, label=right:$s_{n}$] (sn) at (3,-5) {};

\node[fill=white] () at (1,-4.285) {$\cdots$};

\end{tikzpicture}
\end{center}
\caption{\mbox{Diagram for the presentation of Corran-Picantin of $B(e,e,n)$.}}
\end{figure}

The next definition establishes a presentation of the Hecke algebra $H(e,e,n)$ attached to the group $G(e,e,n)$, by using the presentation of Corran-Picantin of the complex braid group $B(e,e,n)$.

\begin{definition}\label{DefPresH(e,e,n)}

Let $e \geq 1$ and $n \geq 2$. We exclude the case ($n=2$, $e$ even), see Remark \ref{RemH(e,e,2)} below. Let $R_0 = \mathbb{Z}[a]$. The unitary associative Hecke algebra $H(e,e,n)$ is defined as the quotient of the group algebra $R_0(B(e,e,n))$ by the following relations:

\begin{enumerate}

\item $t_i^2 - at_i - 1 = 0$ for $i \in \mathbb{Z}/e\mathbb{Z}$,
\item $s_j^2 - as_j - 1 = 0$ for $3 \leq j \leq n$,

\end{enumerate}
where $\{t_i\ |\ i \in \mathbb{Z}/e\mathbb{Z}\} \cup \{s_j\ |\ 3 \leq j \leq n\}$ is the set of generators of the presentation of Corran-Picantin of $B(e,e,n)$. Then, a presentation of $H(e,e,n)$ is obtained by adding these relations to those of the presentation of Corran-Picantin given in Definition \ref{DefPresB(e,e,n)CorranPicantin}.

\end{definition}

\begin{remark}\label{RemH(e,e,2)}

For the case ($n=2$, $e$ even), there exists two conjugacy classes of the reflections $t_i$, for $i \in \mathbb{Z}/e\mathbb{Z}$ in the complex reflection group. In this case, we define the Hecke algebra $H(e,e,2)$ in the same way as in Definition \ref{DefPresH(e,e,n)} over $R_0 = \mathbb{Z}[a_1,a_2]$ with two types of polynomial relations for each conjugacy class of the $t_i$'s: $t_i^2 - a_1t_i - 1 = 0$ for the first conjugacy class and $t_j^2 - a_2t_j - 1 = 0$ for the second.

\end{remark}

Note that we use the polynomial ring $R_0$ instead of the usual Laurent polynomial ring $R$ introduced in Definition \ref{HeckeBMR} and we use normalized polynomial relations in the definition of $H(e,e,n)$. Actually, by a result of Marin (see Proposition 2.3 in \cite{MarinG20G21}) applied to the case of $G(e,e,n)$, the BMR freeness conjecture for this case is equivalent to the fact that $H(e,e,n)$ is a free $R_0$-module of rank equal to the order of $G(e,e,n)$. We will also use a polynomial ring and normalized relations in the definition of the Hecke algebras attached to all the groups $G(de,e,n)$ in the next subsection.

\subsection{The general case}

Corran-Lee-Lee \cite{CorranLeeLee} established a presentation for the complex braid group $B(de,e,n)$ that give rise to nice combinatorial structures (called quasi-Garside structures) for these groups. The presentation is defined as follows.

\begin{definition}\label{DefPresB(de,e,n)CorranLeeLee}

Let $d >1$, $e \geq 1$ and $n \geq 2$. The group $B(de,e,n)$ is defined by a presentation with set of generators: $\{z\} \cup \{t_i\ |\ i \in \mathbb{Z}\} \cup \{s_3, s_4, \cdots, s_n \}$ and relations:

\begin{enumerate}

\item $z t_i = t_{i-e} z$ for $i \in \mathbb{Z}$,
\item $z  s_j=s_j z$ for $3 \leq j \leq n$,
\item $t_i t_{i-1} = t_j t_{j-1}$ for $i, j \in \mathbb{Z}$,
\item $t_i s_3 t_i = s_3 t_i s_3$ for $i \in \mathbb{Z}$,
\item $s_j t_i = t_i s_j$ for $i \in \mathbb{Z}$ and $4 \leq j \leq n$,
\item $s_i s_{i+1} s_i = s_{i+1} s_i s_{i+1}$ for $3 \leq i \leq n-1$,
\item $s_i s_j = s_j s_i$ for $|i-j| > 1$.

\end{enumerate}

\end{definition}

Note that $B(de,e,n)$ is isomorphic to $B(2e,e,n)$ for $d > 1$. The parameter $d$ makes an appearance when it comes to the complex reflection group $G(de,e,n)$. Adding the relation $z^d=1$ and the quadratic relations to all the other generators of the presentation of Corran-Lee-Lee of $B(de,e,n)$, we obtain the presentation of $G(de,e,n)$ given earlier in Definition \ref{DefCorranLeePresG(de,e,n)}. Actually, with the additional relation $z^d=1$, we have $t_{i+de} = z^dt_{i+de} = t_iz^d = t_i$ for all $i \in \mathbb{Z}$.\\

Corran-Lee-Lee proposed in \cite{CorranLeeLee} the following diagram to describe their presentation of $B(de,e,n)$. The nodes are the generators of the presentation. Relation 1 of definition \ref{DefPresB(de,e,n)CorranLeeLee} is described by the curved arrow. Relation 3 is described by the vertical line such that the nodes $t_i$'s are tangent to this line. All the other edges follow the standard conventions of type $A_{n-1}$ Artin groups.

\begin{figure}[H]
$$\begin{xy}
(-2,8.7); (-2,-8.3) **\crv{(-9,0)} ?(0.8) *!/^2mm/{} ?>*@{>};
(-7.5,0) *++={\rule{0pt}{4pt}} *\frm{o};
(-11,0)   *++={z} ;
(1.8,-25) *++={\vdots};
(1.8, 25) *++={\vdots} **@{-};
(4, 16) *++={\rule{0pt}{4pt}} *\frm{o};
    (20,0) *++={\rule{0pt}{4pt}} *\frm{o} **@{-};
(4,-16) *++={\rule{0pt}{4pt}} *\frm{o};
    (20,0) *++={\rule{0pt}{4pt}} *\frm{o} **@{-};
(4, 8) *++={\rule{0pt}{4pt}} *\frm{o};
    (20,0) *++={\rule{0pt}{4pt}} *\frm{o} **@{-};
(4,-8) *++={\rule{0pt}{4pt}} *\frm{o};
    (20,0) *++={\rule{0pt}{4pt}} *\frm{o} **@{-};
(4, 0) *++={\rule{0pt}{4pt}} *\frm{o};
    (20,0) *++={\rule{0pt}{4pt}} *\frm{o} **@{-};
(30, 0) *++={\rule{0pt}{4pt}} *\frm{o} **@{-};
(40, 0) *++={\rule{0pt}{4pt}} *\frm{o} **@{-};
(50, 0) *++={\dots}  **@{-};
(60, 0) *++={\rule{0pt}{4pt}} *\frm{o} **@{-};
(0,-18) *++={};
(6,12) *++={t_2};
(6, 4) *++={t_1};
(6,-4) *++={t_0};
(7,-12)*++={t_{-1}};
(7,-20)*++={t_{-2}};
(23,-3) *++={s_3};
(33,-3) *++={s_4};
(43,-3) *++={s_5};
(63,-3) *++={s_n}
\end{xy}$$
\caption{Diagram for the presentation of Corran-Lee-Lee of $B(de,e,n)$.}\label{DiagramPresCorranLeeLeeB(de,e,n)}
\end{figure}
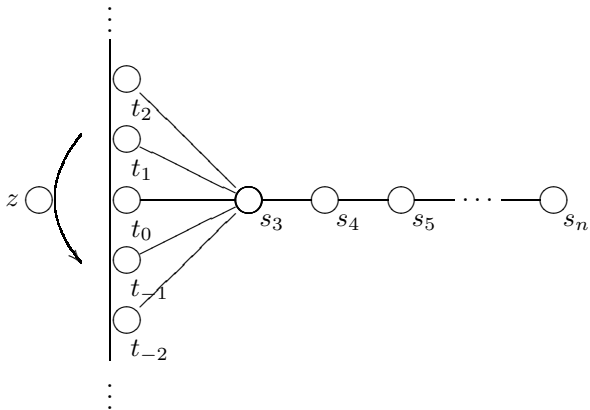

Note that for $e=1$, it is readily checked that the presentation of Corran-Lee-Lee is equivalent to the classical presentation of $B(d,1,n)$ that is isomorphic to $B(2,1,n)$ for $d > 1$. Note that we replace $t_0$ by $s_2$ in the set of generators. The (well-known) diagram that describes the classical presentation of $B(d,1,n)$ is the following.

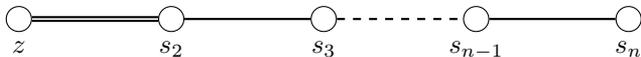
\begin{figure}[H]

\begin{center}
\begin{tikzpicture}

\node[draw, shape=circle, label=below:$z$] (1) at (0,0) {};
\node[draw, shape=circle,label=below:$s_2$] (2) at (2,0) {};
\node[draw, shape=circle,label=below:$s_3$] (3) at (4,0) {};
\node[draw, shape=circle,label=below:$s_{n-1}$] (n-1) at (6,0) {};
\node[draw,shape=circle,label=below:$s_n$] (n) at (8,0) {};

\draw[thick,-,double] (1) to (2);
\draw[thick,-] (2) to (3);
\draw[dashed,-,thick] (3) to (n-1);
\draw[thick,-] (n-1) to (n);

\end{tikzpicture}
\end{center}

\end{figure}

We are ready to establish a presentation of the Hecke algebra $H(de,e,n)$ attached to the group $G(de,e,n)$, by using the presentation of Corran-Lee-Lee of the complex braid group $B(de,e,n)$. Similarly to the case of $H(e,e,n)$, we also define the Hecke algebra $H(de,e,n)$ over a polynomial ring $R_0$ and use normalized polynomial relations.

\begin{definition}\label{DefH(de,e,n)}

Let $d >1$, $e \geq 1$ and $n \geq 2$. We exclude the case ($n=2$, $e$ even), see Remark \ref{RemH(de,e,2)} below. Let $R_0 = \mathbb{Z}[a,b_1,b_2, \cdots, b_{d-1}]$. The unitary associative Hecke algebra $H(de,e,n)$ is defined as the quotient of the group algebra $R_0(B(de,e,n))$ by the following relations:

\begin{enumerate}

\item $z^d -b_1z^{d-1} - b_2 z^{d-2} - \cdots - b_{d-1}z - 1 = 0$,
\item $t_i^2 - at_i - 1 = 0$ for $i \in \mathbb{Z}$,
\item $s_j^2 - as_j - 1 = 0$ for $3 \leq j \leq n$,

\end{enumerate}
where $\{z\} \cup \{t_i\ |\ i \in \mathbb{Z}\} \cup \{s_j\ |\ 3 \leq j \leq n\}$ is the set of generators of the presentation of Corran-Lee-Lee of $B(de,e,n)$. Then, a presentation of $H(de,e,n)$ is obtained by adding these relations to those given in Definition \ref{DefPresB(de,e,n)CorranLeeLee}.

\end{definition}

\begin{remark}\label{RemH(de,e,2)}

When ($n=2$, $e$ even), the Hecke algebra $H(de,e,2)$ can be defined over $R_0[a_1,a_2,b_1,b_2, \cdots , b_{d-1}]$ in the same way as in the previous definition, but with two types of polynomial relations for the $t_i$'s $($due to the existence of two conjugacy classes of the $t_i$'s in $G(de,e,2))$, as established before in Remark \ref{RemH(e,e,2)}.

\end{remark}

\section{Bases for the Hecke algebras $H(e,e,n)$}\label{SectionBasisHecke}

The Hecke algebra $H(e,e,n)$ is described in Definition \ref{DefPresH(e,e,n)} by a presentation with generating set $\{ t_0, t_1, \cdots, t_{e-1}, s_3, \cdots, s_n \}$. It is defined over $R_0 = \mathbb{Z}[a]$. Note that we will replace $t_0$ by $s_2$ in some cases in order to simplify notations. Using the geodesic normal forms of $G(e,e,n)$ introduced in Section \ref{SectionReducedWordsG(de,e,n)}, we construct a natural basis for $H(e,e,n)$ that is different from the one introduced by Ariki in \cite{ArikiHecke}.\\

Let us define the following subsets of $H(e,e,n)$:

\begin{center}
\begin{tabular}{llll}
$\Lambda_2$ & $=$ & $\{1,$\\
& & $\ \ t_k$ & for $0 \leq k \leq e-1$,\\
& & $\ \ t_kt_0$ & for $1 \leq k \leq e-1 \}$,
\end{tabular}
\end{center}
and for $3 \leq i \leq n$,
\begin{center}
\begin{tabular}{llll}
$\Lambda_i$ & $=$ & $\{1$,\\
 & & $\ \ s_i \cdots s_{i'}$ & for $3 \leq i' \leq i$,\\ 
 & & $\ \ s_i \cdots s_{3}t_k$ & for $0 \leq k \leq e-1$,\\
 & & $\ \ s_i \cdots s_{3}t_ks_2 \cdots s_{i'}$ & for $1 \leq k \leq e-1$ and $2 \leq i' \leq i\}.$\\
\end{tabular}
\end{center}

\noindent Define $\Lambda = \Lambda_2 \cdots \Lambda_n$ to be the set of the products $a_2 \cdots a_n$, where $a_2 \in \Lambda_2, \cdots, a_n \in \Lambda_n$. Remark that this set corresponds to all the reduced words $R\!E(w)$ of the form $R\!E_2(w) R\!E_3(w) \cdots R\!E_n(w)$ introduced in Section 3 (see Definition \ref{DefREiwG(de,e,n)} and Remark \ref{GeodesicFormsForG(e,e,n)}). Recall that $R_0 = \mathbb{Z}[a]$ (see Definition \ref{DefPresH(e,e,n)}). The aim of this section is to establish the following.

\begin{theorem}\label{TheoremNewBasis}

The set $\Lambda$ provides an $R_0$-basis of the Hecke algebra $H(e,e,n)$.

\end{theorem}

In order to prove this theorem, it is shown in Proposition 2.3$(i)$ of \cite{MarinG20G21} that it is enough to find a spanning set of $H(e,e,n)$ over $R_0$ of $|G(e,e,n)|$ elements. This is a general fact about Hecke algebras associated to complex reflection groups. We have $|\Lambda_2| = 2e$, $|\Lambda_3| = 3e$, $\cdots$, and $|\Lambda_n| = ne$ by the definition of $\Lambda_2$, $\cdots$, and $\Lambda_n$. Thus, $|\Lambda|$ is equal to $e^{n-1}n!$ that is the order of $G(e,e,n)$. If we manage to prove that $\Lambda$ is a spanning set of $H(e,e,n)$ over $R_0$, then we get Theorem \ref{TheoremNewBasis}. Denote by Span$(S)$ the sub-$R_0$-module of $H(e,e,n)$ generated by $S$.\\

We prove Theorem \ref{TheoremNewBasis} by induction on $n \geq 2$. Propositions \ref{PropInductionHypothesis} and \ref{PropCasen=3} below correspond to the cases $n=2$ and $n=3$, respectively. Suppose that $\Lambda_2 \cdots \Lambda_{n-1}$ is an $R_0$-basis of $H(e,e,n-1)$. As mentioned before, in order to prove that $\Lambda = \Lambda_2 \cdots \Lambda_{n}$ is an $R_0$-basis of $H(e,e,n)$, it is enough to show that it is an $R_0$-generating set of $H(e,e,n)$, that is $\Lambda$ stable under left multiplication by $t_0, \cdots , t_{e-1}, s_3$, $\cdots$, $s_n$. Since $\Lambda_2 \cdots \Lambda_{n-1}$ is an $R_0$-basis of $H(e,e,n-1)$, the set $\Lambda_2 \cdots \Lambda_n$ is stable under left multiplication by $t_0, \cdots , t_{e-1},s_3$, $\cdots$, $s_{n-1}$. We prove that it is stable under left multiplication by $s_n$, that is $s_n(a_2 \cdots a_n) = a_2 \cdots a_{n-2}s_n(a_{n-1}a_n)$ belongs to Span$(\Lambda)$ for $a_2 \in \Lambda_2$, $\cdots$, $a_n \in \Lambda_n$ by checking all the different possibilities for $a_{n-1} \in \Lambda_{n-1}$ and $a_n \in \Lambda_n$.\\

Assume $n > 3$. If $a_{n-1} =1$ or $a_n =1$, it is readily checked that $s_n(a_{n-1}a_n)$ belongs to Span$(\Lambda_{n-1}\Lambda_n)$. If $a_{n-1} = s_{n-1} \cdots s_{i}$ for $2 \leq i \leq n-1$, we distinguish $3$ different cases for $a_n$ that belongs to $\Lambda_n$. This is done in Lemmas \ref{Lemma2}, \ref{Lemma3} and \ref{Lemma4} below. If $a_{n-1} = s_{n-1} \cdots s_{3}t_k$ for $0 \leq k \leq e-1$, we also distinguish $3$ different cases for $a_n \in \Lambda_n$. This is done in Lemmas \ref{Lemma5}, \ref{Lemma6} and \ref{Lemma7}. Finally, if $a_{n-1} = s_{n-1} \cdots s_{3}t_ks_2 \cdots s_{i}$ for $1 \leq k \leq e-1$ and $2 \leq i \leq n-1$, we also have $3$ different cases for $a_n \in \Lambda_n$ (see Lemmas \ref{Lemma8}, \ref{Lemma9} and \ref{Lemma10} below).\\

Let us start by establishing the following two preliminary lemmas.

\begin{lemma}\label{LemmaTLTKInVect}

For $i, j \in \mathbb{Z}/e\mathbb{Z}$, we have $t_j t_i \in$ Span$(\Lambda_2)$.

\end{lemma}

\begin{proof}

If $i=j$, then we have $t_i^2 = at_i + 1 \in$ Span$(\Lambda_2)$. Suppose $i \neq j$. We have\\
$t_i t_{i-1} = t_j t_{j-1}$. If we multiply by $t_j$ on the left and by $t_{i-1}$ on the right, we get\\
$t_j t_i t_{i-1}^2 = t_j^{2} t_{j-1}t_{i-1}$. Using the quadratic relations, we have\\
$t_jt_i(a t_{i-1} +1) = (a t_j + 1)t_{j-1}t_{i-1}$, that is\\
$a t_j t_i t_{i-1} + t_jt_i = a t_j t_{j-1} t_{i-1} + t_{j-1}t_{i-1}$.\\
Replacing $t_i t_{i-1}$ by $t_jt_{j-1}$ and  $t_jt_{j-1}$ by $t_i t_{i-1}$, we get\\
$at_j^2 t_{j-1} + t_jt_i = at_it_{i-1}^2 + t_{j-1}t_{i-1}$. Using the quadratic relations, we have\\
$a(at_j +1) t_{j-1} + t_jt_i = at_i(at_{i-1}+1) + t_{j-1}t_{i-1}$, that is\\
$a^2 t_1t_0 + at_{j-1} + t_jt_i = a^2t_1t_0 + a t_i + t_{j-1}t_{i-1}$. Simplifying this relation, we get
\begin{center} $t_jt_i = t_{j-1}t_{i-1} + a(t_i - t_{j-1})$.\end{center}
Now, we apply the same operations to compute $t_{j-1}t_{i-1}$ and so on until we arrive to a term of the form $t_kt_0$ for some $k \in \mathbb{Z}/e\mathbb{Z}$. Thus, if $i \neq j$, then $t_jt_i$ belongs to Span$(\Lambda_2 \setminus \{1\})$.

\end{proof}

\begin{lemma}\label{LemmaTkT0}

For $1 \leq k \leq e-1$, we have $t_kt_0 \in R_0(t_1t_0) + R_0(t_1t_0)^2 + \cdots + R_0(t_1t_0)^k + R_0t_1 + R_0t_2 + \cdots + R_0t_{k-1}$.

\end{lemma}

\begin{proof}

We prove the property by induction on $k$. The property is clearly satisfied for $k = 1$. Let $k \geq 2$. Suppose $t_{k-1}t_0 \in R_0(t_1t_0) + R_0(t_1t_0)^2 + \cdots + R_0(t_1t_0)^{k-1} + R_0t_1 + R_0t_2 + \cdots + R_0t_{k-2}$. We have that $t_{k+1}t_k = t_kt_{k-1}$. Multiplying by $t_{k+1}$ on the left and by $t_0$ on the right, we get $t_{k+1}^2t_kt_0 =t_{k+1} t_kt_{k-1}t_0$. Using the quadratic relations and replacing $t_{k+1} t_k$ with $t_1t_0$, we get $(at_{k+1} +1)t_kt_0 =t_1 t_0t_{k-1}t_0$. After simplifying this relation, we have $t_kt_0 = t_1t_0(t_{k-1}t_0) - a(t_1t_0)t_0$. Using the induction hypothesis, we replace $t_{k-1}t_0$ by its value and we get $t_kt_0 \in t_1t_0(R_0(t_1t_0) + R_0(t_1t_0)^2 + \cdots + R_0(t_1t_0)^{k-1} + R_0t_1 + R_0t_2 + \cdots + R_0t_{k-2}) + R_0(t_1t_0)t_0$. This is equal to $R_0(t_1t_0)^2 + R_0(t_1t_0)^3 + \cdots + R_0(t_1t_0)^{k} + R_0(t_1t_0)t_1 + R_0(t_1t_0)t_2 + \cdots + R_0(t_1t_0)t_{k-2} + R_0(t_1t_0)t_0$. Now $(t_1t_0)t_m$ is equal to $t_{m+1}t_mt_m \in R_0t_1t_0 + R_0t_{m+1}$ for $1 \leq m \leq k-2$ and $(t_1t_0)t_0 \in R_0(t_1t_0) + R_0t_1$. It follows that $t_kt_0 \in R_0(t_1t_0) + R_0(t_1t_0)^2 + \cdots + R_0(t_1t_0)^k + R_0t_1 + R_0t_2 + \cdots + R_0t_{k-1}$.

\end{proof}

As a direct consequence of Lemma \ref{LemmaTLTKInVect}, we have the following.

\begin{proposition}\label{PropInductionHypothesis}

Let $x = t_l$ with $l \in \mathbb{Z}/e\mathbb{Z}$. For all $a_2 \in \Lambda_2$, we have $xa_2$ belongs to Span$(\Lambda_2)$.

\end{proposition}

\begin{remark}\label{RemarkTwoConjugHecke}

Recall that we excluded the case ($n = 2$, $e$ even) in Definition \ref{DefPresH(e,e,n)} (see Remark \ref{RemH(e,e,2)}). Consider the Hecke algebra $H(e,e,2)$ with $e$ even. Similarly to Proposition \ref{PropInductionHypothesis}, one shows that $\Lambda_2$ is stable under left multiplication by all the $t_i$'s for $i \in \mathbb{Z}/e\mathbb{Z}$. Then, applying Proposition (2.3) (i) of \cite{MarinG20G21}, we also get that $\Lambda_2$ is an $R_0$-basis of $H(e,e,2)$ for the case $e$ even.

\end{remark}

\begin{proposition}\label{PropCasen=3}

For all $a_2 \in \Lambda_2$ and $a_3 \in \Lambda_3$, the element $s_3(a_2a_3)$ belongs to Span$(\Lambda_2\Lambda_3)$.

\end{proposition}

\begin{proof}

The case where $a_1 \in \Lambda_1$ and $a_2 = 1$ is obvious. The case where $a_1 = 1$ and $a_2 \in \Lambda_2$ is also obvious.

\emph{Case 1}. Suppose $a_2 = t_k$ for $0 \leq k \leq e-1$ and $a_3 = s_3$.\\
We have $s_3t_ks_3 = t_k\underline{s_3t_k} \in$ Span$(\Lambda_2\Lambda_3)$.

\emph{Case 2}. Suppose $a_2 = t_k$ for $0 \leq k \leq e-1$ and $a_3 = s_3t_l$ for $0 \leq l \leq e-1$.\\
We have $\underline{s_3t_ks_3}t_l = t_ks_3t_kt_l$. After replacing $t_kt_l$ by its decomposition over $\Lambda_2$ (see Lemma \ref{LemmaTLTKInVect}), we directly have $s_3t_ks_3t_l \in$ Span$(\Lambda_2\Lambda_3)$.

\emph{Case 3}. Suppose $a_2 = t_k$ for $0 \leq k \leq e-1$ and $a_3 = s_3t_lt_0$ or $a_3 = s_3t_lt_0s_3$ for $1 \leq l \leq e-1$. We have $\underline{s_3t_ks_3}t_lt_0s_3 = t_ks_3t_kt_lt_0s_3$. By replacing $t_kt_l$ by its value (see Lemma \ref{LemmaTLTKInVect}), we obviously have $s_3t_ks_3t_lt_0$ and $s_3t_ks_3t_lt_0s_3$ belong to Span$(\Lambda_2\Lambda_3)$.

\emph{Case 4}. Suppose $a_2 = t_kt_0$ for $1 \leq k \leq e-1$ and $a_3 = s_3$. We have $s_3(a_2a_3) = s_3t_kt_0s_3 \in$ Span$(\Lambda_2\Lambda_3)$.

\emph{Case 5}. Suppose $a_2 = t_kt_0$ with $1 \leq k \leq e-1$ and $a_3 = s_3t_l$ with $0 \leq l \leq e-1$.\\
We have $s_3(a_2a_3) = s_3t_kt_0s_3t_l$. Recall that by Lemma \ref{LemmaTkT0}, we have $t_kt_0 \in R_0(t_1t_0) + R_0(t_1t_0)^2 + \cdots + R_0(t_1t_0)^k + R_0t_1 + R_0t_2 + \cdots + R_0t_{k-1}$. Replacing $t_kt_0$ by its value, we have to deal with the following two terms:\\
$s_3 t_x s_3t_l$ with $1 \leq x \leq k-1$ and
$s_3 (t_1t_0)^x s_3t_l$ with $1 \leq x \leq k$.\\
The first term is done in Case 2. For the second term, we decrease the power of $(t_1t_0)$ and use $t_1t_0 = t_{l+1}t_l$ to get
$s_3 (t_1t_0)^{x-1}t_{l+1}\underline{t_l s_3t_l}$. We apply a braid relation and then get
$s_3 (t_1t_0)^{x-1}t_{l+1}s_3t_ls_3$. Again, we decrease the power of $(t_1t_0)$ and use $t_1t_0 = t_{l+2}t_{l+1}$. We get $s_3 (t_1t_0)^{x-2}t_{l+2}t_{l+1}^2s_3t_ls_3 \in R_0 s_3(t_1t_0)^{x-2}t_{l+2}t_{l+1}s_3t_ls_3 + R_0s_3 (t_1t_0)^{x-2}t_{l+2}s_3t_ls_3$. We continue by decreasing the power of $(t_1t_0)$ and we get in the next step that $s_3(a_2 a_3)$ belongs to
\begin{small}
$R_0 s_3(t_1t_0)^{x-3}t_{l+1}s_3t_{l}s_3^3$
$+ R_0 s_3(t_1t_0)^{x-3}t_{l+2}s_3t_{l}s_3^2$\\
$+ R_0 s_3 (t_1t_0)^{x-3}t_{l+1}s_3t_{l}s_3^2$
$+ R_0 s_3 (t_1t_0)^{x-3}t_{l+3}s_3t_ls_3$.\end{small}
Inductively, we arrive to terms of the form
$s_3 t_1t_0t_{x'}s_3t_l(s_3)^{x''}$ ($0 \leq x' \leq e-1$ and $x'' \in \mathbb{N}$). Replace $t_1t_0$ by $t_{x'+1}t_{x'}$, we get $s_3 t_{x'+1}(t_{x'})^2s_3t_l(s_3)^{x''}$ which belongs to
$R_0 s_3 t_{x'+1}t_{x'}s_3t_l(s_3)^{x''}$
$+ R_0 s_3 t_{x'+1}s_3t_l(s_3)^{x''}$. Replacing $t_{x'+1}t_{x'}$ by $t_{l+1}t_l$ and applying a braid relation in the first term, we get
$R_0 s_3 t_{l+1}s_3t_l(s_3)^{x''+1}$
$+ R_0 s_3 t_{x'+1}s_3t_l(s_3)^{x''}$.
Since $s_3^2=as_3+1$, it remains to deal with these $2$ terms:
\begin{center}$s_3 t_xs_3t_l$ and $s_3 t_xs_3t_ls_3$, for some $0 \leq x \leq e-1$.\end{center} It is readily checked that these terms belong to Span$(\Lambda_2 \Lambda_3)$.

\emph{Case 6}. Suppose $a_2 = t_kt_0$ and $a_3 = s_3t_lt_0$ or $a_3 = s_3t_lt_0s_3$ ($1 \leq k,l \leq e-1$).\\
By Case 5, we get two terms of the form $s_3 t_xs_3t_l$ and $s_3 t_xs_3t_ls_3$. Multiplying them on the right by $t_0$ then by $s_3$, we get that $s_3(a_2a_3)$ obviously belongs to Span$(\Lambda_2 \Lambda_3)$.

\end{proof}

In the sequel, we will indicate by (1) the operation that shifts the underlined letters to the left, by (2) the operation that applies braid relations, and by (3) the one that applies quadratic relations. The following lemma is useful in the proofs of Lemmas \ref{Lemma7}, \ref{Lemma9} and \ref{Lemma10} below. Denote by $S_{n-1}^{*}$ the set of the words over $\{t_0, \cdots, t_{e-1}, s_3, \cdots, s_{n-1} \}$.\\

\begin{lemma}\label{LmScholie1}

Let $3 \leq i \leq n$. We have $s_n \cdots s_4s_3^2 s_4 \cdots s_i$ belongs to Span$(S_{n-1}^{*}\Lambda_n)$. 

\end{lemma}

\begin{proof}

If $i=3$, we have $s_n \cdots s_4s_3^2 \in R_0 s_n \cdots s_4s_3 + R_0 s_n \cdots s_4$ (the last term is equal to $1$ if $n=3$). If $i = 4$, we have $s_n \cdots s_4 s_3^2s_4 \in R_0 s_n \cdots s_4s_3s_4 + R_0 s_n \cdots s_4^2 \in R_0 s_3 s_n \cdots s_3 + R_0 s_n \cdots s_4 + R_0 s_n \cdots s_5$. The last term is equal to $1$ if $n = 4$. Let $i \geq 5$.
We have $s_n \cdots s_4s_3^2 s_4 \cdots s_i$ belongs to $R_0 s_n \cdots s_4 s_3 s_4 \cdots s_i + R_0 s_n \cdots s_5 s_4^2 s_5 \cdots s_i$. We apply the quadratic relation $s_4^2 = as_4+1$ to the second term and get\\
$R_0 s_n \cdots s_5 s_4 s_5 \cdots s_i + R_0 s_n \cdots s_6 s_5^2 s_6 \cdots s_i$. And so on, we apply quadratic relations. We get terms of the form $s_n \cdots s_{k+1}s_ks_{k+1} \cdots s_i$ with $k+1 \leq i$ and a term of the form $s_n \cdots s_{i+1}s_is_{i-1}^{2}s_i$.\\
We have $s_n \cdots s_{i+1}s_is_{i-1}^{2}s_i$ belongs to\\
$R_0 s_n \cdots s_{i+1}s_is_{i-1}s_i + R_0 s_n \cdots s_{i+1}s_i + R_0 s_n \cdots s_{i+1} \subseteq \\ R_0 s_{i-1} s_n \cdots s_{i-1} + R_0 s_n \cdots s_{i+1}s_i + R_0 s_n \cdots s_{i+1}$ (the last term is equal to $1$ if $i = n$). Hence it belongs to Span$(S_{n-1}^{*}\Lambda_n)$.\\
The other terms are of the form $s_n \cdots s_{k+1}s_ks_{k+1} \cdots s_i$ ($k+1 \leq i$). We have
$\begin{array}{ll}
s_n \cdots \underline{s_{k+1}s_ks_{k+1}} \cdots s_i & \stackrel{(2)}{=}\\
s_n \cdots s_{k+2}\underline{s_k} s_{k+1}s_k \underline{s_{k+2}} \cdots s_{i} & \stackrel{(1)}{=}\\
s_k s_n \cdots \underline{s_{k+2} s_{k+1}s_{k+2}}s_k s_{k+3} \cdots s_{i} & \stackrel{(2)}{=}\\
s_k s_n \cdots \underline{s_{k+1}} s_{k+2}s_{k+1}s_k s_{k+3} \cdots s_{i} & \stackrel{(1)}{=}
\end{array}$\\
$\begin{array}{l}
s_k s_{k+1} s_n \cdots s_{k+2}s_{k+1}s_k \underline{s_{k+3}} \cdots \underline{s_{i-1}} s_{i}.
\end{array}$\\
We apply the same operations to $\underline{s_{k+3}}$, $\cdots$, $\underline{s_{i-1}}$ and get\\
$\begin{array}{ll}
s_k s_{k+1} \cdots s_{i-2} s_n \cdots s_{k} \underline{s_{i}} & \stackrel{(1)}{=}\\
s_k s_{k+1} \cdots s_{i-2} s_n \cdots \underline{s_i s_{i-1}s_i} s_{i-2} \cdots s_k & \stackrel{(2)}{=}\\
s_k s_{k+1} \cdots s_{i-2} s_n \cdots s_{i+1}\underline{s_{i-1}} s_{i}s_{i-1} s_{i-2} \cdots s_k & \stackrel{(1)}{=}\\
s_ks_{k+1} \cdots s_{i-1} s_ns_{n-1} \cdots s_{k}, &
\end{array}$\\
which belongs to Span$(S_{n-1}^{*}\Lambda_n)$.

\end{proof}

\begin{lemma}\label{Lemma2}

If $a_{n-1} = s_{n-1}s_{n-2} \cdots s_i$ with $3 \leq i \leq n-1$ and $a_n = s_n s_{n-1} \cdots s_{i'}$ with $3 \leq i' \leq n$, then $s_n (a_{n-1}a_n)$ belongs to Span$(\Lambda_{n-1}\Lambda_n)$.

\end{lemma}

\begin{proof}

\emph{Suppose $i < i'$.} We have $s_n(a_{n-1}a_n)$ is equal to\\
$\begin{array}{ll}
s_ns_{n-1}s_{n-2} \cdots s_i\underline{s_n}s_{n-1} \cdots s_{i'} & \stackrel{(1)}{=}\\
\underline{s_ns_{n-1}s_n}s_{n-2} \cdots s_{i}s_{n-1} \cdots s_{i'} & \stackrel{(2)}{=}\\
s_{n-1}s_ns_{n-1}s_{n-2} \cdots s_{i}\underline{s_{n-1}} \cdots s_{i'} & \stackrel{(1)}{=}\\
s_{n-1}s_n\underline{s_{n-1}s_{n-2}s_{n-1}} \cdots s_{i}s_{n-2} \cdots s_{i'} & \stackrel{(2)}{=}\\
s_{n-1}s_n\underline{s_{n-2}}s_{n-1}s_{n-2} \cdots s_{i}s_{n-2} \cdots s_{i'} & \stackrel{(1)}{=} 
\end{array}$\\
$\begin{array}{l}
s_{n-1}s_{n-2}s_ns_{n-1}s_{n-2} \cdots s_{i}\underline{s_{n-2}} \cdots \underline{s_{i'}}.
\end{array}$\\
We apply the same operations to the underlined letters $\underline{s_{n-2}} \cdots \underline{s_{i'}}$ in order to get
$s_{n-1}s_{n-2}\cdots s_{i'-1}s_ns_{n-1} \cdots s_i$ which belongs to Span$(\Lambda_{n-1}\Lambda_n)$.\\

\emph{Suppose $i \geq i'$.} We have $s_n(a_{n-1}a_n)$ is equal to $s_ns_{n-1} \cdots s_i\underline{s_n}s_{n-1} \cdots s_{i'}$. We apply the operations (1) and (2) and get $s_{n-1}s_{n}s_{n-1} \cdots s_{i}\underline{s_{n-1}} \cdots s_{i'}$. Then we apply the same operations to $\underline{s_{n-1}}$ and get $s_{n-1}s_{n-2}s_ns_{n-1} \cdots s_i \underline{s_{n-2} \cdots s_{i'}}$.
Since $i \geq i'$, we write $s_{i+2}s_{i+1}s_i$ in $\underline{s_{n-2} \cdots s_{i'}}$ and get $s_{n-1}s_{n-2}s_ns_{n-1} \cdots s_i \underline{s_{n-2}} \cdots \underline{s_{i+2}}s_{i+1}s_i \cdots s_{i'}$. Similarly, we apply the same operations to $\underline{s_{n-2}}$, $\cdots$, $\underline{s_{i+2}}$ and get\\
$\begin{array}{ll}
s_{n-1} \cdots s_{i+1}s_n \cdots \underline{s_{i+1}s_is_{i+1}}s_{i}\cdots s_{i'} & \stackrel{(2)}{=}\\
s_{n-1} \cdots s_{i+1}s_n \cdots s_{i+2}\underline{s_{i}}s_{i+1}s_{i}^2s_{i-1} \cdots s_{i'} & \stackrel{(1)}{=}\\
s_{n-1} \cdots s_{i}s_n \cdots s_{i+2}s_{i+1}\underline{s_{i}^2}s_{i-1} \cdots s_{i'} & \stackrel{(3)}{=}
\end{array}$\\
$\begin{array}{l} a s_{n-1} \cdots s_{i}s_n \cdots s_{i+1}s_is_{i-1}\cdots s_{i'} +  s_{n-1} \cdots s_{i}s_n \cdots s_{i+2}s_{i+1}\underline{s_{i-1}}\cdots \underline{s_{i'}}.\end{array}$\\
The first term belongs to Span$(\Lambda_{n-1}\Lambda_n)$. For the second term, the underlined letters commute with $s_n \cdots s_{i+2}s_{i+1}$ hence they are shifted to the left. We thus get $s_n(a_{n-1}a_n)$ is equal to
$a s_{n-1} \cdots s_{i}s_n \cdots s_{i'} +  s_{n-1} \cdots s_{i'} s_n \cdots s_{i+1}$ which belongs to Span$(\Lambda_{n-1}\Lambda_n)$.

\end{proof}

\begin{lemma}\label{Lemma3}

If $a_{n-1} = s_{n-1} \cdots s_i$ with $3 \leq i \leq n-1$ and $a_n = s_n \cdots s_{3}t_{k}$ with $0 \leq k \leq e-1$, then $s_n (a_{n-1}a_n)$ belongs to Span$(\Lambda_{n-1}\Lambda_n)$.

\end{lemma}

\begin{proof}

This corresponds to the case $i'=3$ in the proof of Lemma \ref{Lemma2} with a right multiplication by $t_k$ for $0 \leq k \leq e-1$. Since $i \geq 3$, by the case \mbox{$i \geq i'$} of Lemma \ref{Lemma2}, we have $s_n(a_{n-1}a_n) = a s_{n-1} \cdots s_{i}s_n \cdots s_{3}t_k +  s_{n-1} \cdots s_{3} s_n \cdots s_{i+1} t_k$. In the second term, $t_k$ commutes with $s_n \cdots s_{i+1}$ hence it is shifted to the left. We get $s_n(a_{n-1}a_n) = a s_{n-1} \cdots s_{i}s_n \cdots s_{3}t_k +  s_{n-1} \cdots s_{3}t_k s_n \cdots s_{i+1}$ which belongs to Span$(\Lambda_{n-1}\Lambda_n)$.

\end{proof}

\begin{lemma}\label{Lemma4}

If $a_{n-1} = s_{n-1} \cdots s_i$ with $3 \leq i \leq n-1$ and $a_n = s_n \cdots s_{3}t_{k}s_2s_3 \cdots s_{i'}$ with $1 \leq k \leq e-1$ and $2 \leq i' \leq n$, then $s_n (a_{n-1}a_n)$ belongs to Span$(\Lambda_{n-1}\Lambda_n)$.

\end{lemma}

\begin{proof}

According to Lemma \ref{Lemma3}, we have\\ $s_n(a_{n-1}a_n) = a s_{n-1} \cdots s_{i}s_n \cdots s_{3}t_ks_2 \cdots s_{i'} +  s_{n-1} \cdots s_{3}t_k s_n \cdots s_{i+1} s_2 \cdots s_{i'}$.
The first term is an element of Span$(\Lambda_{n-1}\Lambda_n)$. We check that the second term also belongs to Span$(\Lambda_{n-1}\Lambda_n)$. Actually,

\emph{If $i' < i$}, the second term is equal to
$s_{n-1} \cdots s_{3}t_k s_n \cdots s_{i+1} \underline{s_2} \cdots \underline{s_{i'}}$.\\ The underlined letters commute with $s_n \cdots s_{i+1}$ and are shifted to the left. We get $s_{n-1} \cdots s_{3}t_ks_2 \cdots s_{i'}s_n \cdots s_{i+1} \in$ Span$(\Lambda_{n-1}\Lambda_n)$.

\emph{If $i' \geq i$}, we write $s_{i-1}s_is_{i+1}$ in $s_2 \cdots s_{i'}$ and get $s_{n-1} \cdots s_{3}t_k s_n \cdots s_{i+1} (s_2 \cdots s_{i'}) = $\\
$\begin{array}{ll}
s_{n-1} \cdots s_{3}t_k s_n \cdots s_{i+1} (\underline{s_2} \cdots \underline{s_{i-1}}s_is_{i+1} \cdots s_{i'}) & \stackrel{(1)}{=}\\
s_{n-1} \cdots s_{3}t_k s_2 \cdots s_{i-1} s_n \cdots \underline{s_{i+1} (s_is_{i+1}} \cdots s_{i'}) & \stackrel{(2)}{=}\\
s_{n-1} \cdots s_{3}t_k s_2 \cdots s_{i-1} s_n \cdots \underline{s_{i}} s_{i+1} s_{i} (s_{i+2} \cdots s_{i'}) & \stackrel{(1)}{=}\\
s_{n-1} \cdots s_{3}t_k s_2 \cdots s_{i-1}s_i s_n \cdots s_{i+1} s_{i} (\underline{s_{i+2}} \cdots s_{i'}) & \stackrel{(1)}{=}\\
s_{n-1} \cdots s_{3}t_k s_2 \cdots s_{i-1}s_i s_n \cdots \underline{s_{i+2}s_{i+1} s_{i+2}} s_{i}(s_{i+3} \cdots s_{i'}) & \stackrel{(2)}{=}\\
s_{n-1} \cdots s_{3}t_k s_2 \cdots s_{i-1}s_i s_n \cdots \underline{s_{i+1}}s_{i+2} s_{i+1} s_{i}(s_{i+3} \cdots s_{i'}) & \stackrel{(1)}{=}
\end{array}$\\
$\begin{array}{l}
s_{n-1} \cdots s_{3}t_k s_2 \cdots s_i s_{i+1} s_n \cdots s_{i+2} s_{i+1} s_{i}(\underline{s_{i+3}} \cdots \underline{s_{i'}}).
\end{array}$\\
We apply the same operations to the underlined letters $\underline{s_{i+3}}$, $\cdots$, $\underline{s_{i'}}$. We finally get $s_{n-1} \cdots s_{3}t_ks_2 \cdots s_{i'-1}s_n \cdots s_i \in$ Span$(\Lambda_{n-1}\Lambda_n)$.

\end{proof}

\begin{lemma}\label{Lemma5}

If $a_{n-1} = s_{n-1} \cdots s_3t_k$ with $0 \leq k \leq e-1$ and $a_n = s_n \cdots s_{i}$ with $3 \leq i \leq n$, then $s_n (a_{n-1}a_n)$ belongs to Span$(\Lambda_{n-1}\Lambda_n)$.

\end{lemma}

\begin{proof}

We have $s_n(a_{n-1}a_n)$ is equal to\\
$\begin{array}{ll}
s_ns_{n-1} \cdots s_3t_k \underline{s_n} \cdots s_i & \stackrel{(1)}{=}\\
\underline{s_{n}s_{n-1}s_n} \cdots s_3t_ks_{n-1} \cdots s_i  & \stackrel{(2)}{=}\\
s_{n-1}s_{n}s_{n-1} \cdots s_3t_k \underline{s_{n-1}} \cdots s_i & \stackrel{(1)}{=}\\
s_{n-1}s_{n}\underline{s_{n-1}s_{n-2}s_{n-1}} \cdots s_3t_k s_{n-2} \cdots s_i & \stackrel{(2)}{=}\\
s_{n-1}s_{n} \underline{s_{n-2}}s_{n-1}s_{n-2} \cdots s_3t_k s_{n-2} \cdots s_i & \stackrel{(1)}{=}\\
\end{array}$\\
$\begin{array}{l}
s_{n-1} s_{n-2} s_{n}s_{n-1} \cdots s_3t_k \underline{s_{n-2}} \cdots \underline{s_i}.
\end{array}$\\ Now we apply the same operations for $\underline{s_{n-2}}$, $\cdots$, $\underline{s_i}$.

\emph{If $i=3$}, we get $s_{n-1} \cdots s_3 s_{n}s_{n-1} \cdots \underline{s_3t_ks_3}$. Next, we apply a braid relation to get $s_{n-1} \cdots s_3 s_{n}s_{n-1} \cdots \underline{t_k}s_3t_k$, then we shift $\underline{t_k}$ to the left and we finally get $s_{n-1} \cdots s_3 t_k s_{n}s_{n-1} \cdots s_3t_k$ which belongs to Span$(\Lambda_{n-1}\Lambda_n)$.

\emph{If $i > 3$}, we directly get 
$ s_{n-1} \cdots s_is_{i-1} s_{n} \cdots s_3t_k$ that also belongs to \begin{small}Span$(\Lambda_{n-1}\Lambda_n)$.\end{small}

\end{proof}

\begin{lemma}\label{Lemma6}

If $a_{n-1} = s_{n-1} \cdots s_3t_k$ with $0 \leq k \leq e-1$ and $a_n = s_n \cdots s_{3}t_l$ with $0 \leq l \leq e-1$, then $s_n (a_{n-1}a_n)$ belongs to Span$(\Lambda_{n-1}\Lambda_n)$.

\end{lemma}

\begin{proof}

By Lemma \ref{Lemma5}, one can write
$s_n(a_{n-1}a_n) = s_{n-1} \cdots s_3 t_k s_{n}s_{n-1} \cdots s_3t_kt_l$.
By Lemma \ref{LemmaTLTKInVect} where we compute $t_kt_l$, we directly deduce that $s_n(a_{n-1}a_n)$ belongs to Span$(\Lambda_{n-1}\Lambda_n)$.

\end{proof}

\begin{lemma}\label{Lemma7}

If $a_{n-1} = s_{n-1} \cdots s_3t_k$ with $0 \leq k \leq e-1$ and $a_n = s_n \cdots s_{3}t_l s_2 \cdots s_i$ with $2 \leq i \leq n$ and $1 \leq l \leq e-1$, then $s_n (a_{n-1}a_n)$ belongs to Span$(S_{n-1}^{*} \Lambda_n)$.

\end{lemma}

\begin{proof}

By the previous lemma, we have \begin{small}$s_n(a_{n-1}a_n) = s_{n-1}\cdots s_3t_ks_n \cdots s_3 t_k t_l (s_2 \cdots s_i)$.\end{small}
By Lemma \ref{LemmaTLTKInVect}, the case $i = 2$ is obvious. Suppose $i \geq 3$.
After replacing $t_kt_l$ by its value given in Lemma \ref{LemmaTLTKInVect}, we have two different terms in $s_n(a_{n-1}a_n)$ of the form $s_{n-1}\cdots s_3t_ks_n \cdots s_3 t_x (s_2 \cdots s_i)$ with $0 \leq x \leq e-1$ and of the form\\ $s_{n-1}\cdots s_3t_ks_n \cdots s_3 t_x (s_3 \cdots s_i)$ with $0 \leq x \leq e-1$.

For terms of the form $s_{n-1}\cdots s_3t_ks_n \cdots s_3 t_x (s_3 \cdots s_i)$ with $0 \leq x \leq e-1$, we have\\
$\begin{array}{ll}
s_{n-1} \cdots s_3 t_k s_n \cdots \underline{s_3 t_x (s_3} \cdots s_i) & \stackrel{(2)}{=}\\
s_{n-1} \cdots s_3 t_k s_n \cdots \underline{t_x} s_3 t_x(s_4 \cdots s_i) & \stackrel{(1)}{=}\\
s_{n-1} \cdots s_3 t_kt_x s_n \cdots s_3 t_x(\underline{s_4} \cdots s_i) & \stackrel{(1)}{=}\\
s_{n-1} \cdots s_3 t_kt_x s_n \cdots \underline{s_4 s_3 s_4}t_x (s_5\cdots s_i) & \stackrel{(2)}{=}\\
s_{n-1} \cdots s_3 t_kt_x s_n \cdots \underline{s_3} s_4 s_3 t_x (s_5\cdots s_i) & \stackrel{(1)}{=}\\
\end{array}$\\
$\begin{array}{l}
s_{n-1} \cdots s_3 t_kt_x s_3 s_n \cdots s_3 t_x (\underline{s_5} \cdots \underline{s_{i}}).
\end{array}$\\
We apply the same operations for the underlined letters to get\\ $s_{n-1} \cdots s_3 t_kt_xs_3 \cdots s_{i-1} \underline{s_n \cdots s_3t_x}$ which belongs to Span$(S_{n-1}^{*} \Lambda_n)$.

Consider the terms of the form $s_{n-1}\cdots s_3t_ks_n \cdots s_3 t_x (s_2 \cdots s_i)$ with $0 \leq x \leq e-1$.\\ If $x \neq 0$, they belong to Span$(\Lambda_{n-1}\Lambda_n)$.\\ If $x =0$, we have \begin{small}$s_{n-1}\cdots s_3t_ks_n \cdots s_3 t_0 (s_2s_3 \cdots s_i) \in R_0 s_{n-1}\cdots s_3t_ks_n \cdots s_3 t_0 s_3 \cdots s_i + R_0 s_{n-1}\cdots s_3t_ks_n \cdots s_4s_3^2s_4 \cdots s_i$.\end{small} The first term correspond to the previous case (with $x=0$) and then belongs to Span$(S_{n-1}^{*} \Lambda_n)$. By Lemma \ref{LmScholie1}, the second term also belongs to Span$(S_{n-1}^{*} \Lambda_n)$.

\end{proof}

\begin{lemma}\label{Lemma8}

If $a_{n-1} = s_{n-1} \cdots s_3t_ks_2 \cdots s_i$ with $2 \leq i \leq n-1$, $1 \leq k \leq e-1$ and $a_n = s_n \cdots s_{i'}$ with $3 \leq i' \leq n$, then $s_n (a_{n-1}a_n)$ belongs to Span$(\Lambda_{n-1}\Lambda_n)$.

\end{lemma}

\begin{proof}

\emph{Suppose $i < i'$}. We have $s_n(a_{n-1}a_n)$ is equal to\\
$\begin{array}{ll}
s_n \cdots s_3t_ks_2 \cdots s_i \underline{s_n} \cdots s_{i'} & \stackrel{(1)}{=}\\
\underline{s_n s_{n-1}s_n} \cdots s_3t_ks_2 \cdots s_i s_{n-1} \cdots s_{i'} & \stackrel{(2)}{=}\\
s_{n-1} s_{n}s_{n-1} \cdots s_3t_ks_2 \cdots s_i \underline{s_{n-1}} \cdots s_{i'} & \stackrel{(1)}{=}\\
s_{n-1} s_{n}\underline{s_{n-1}s_{n-2}s_{n-1}} \cdots s_3t_ks_2 \cdots s_i s_{n-2} \cdots s_{i'} & \stackrel{(2)}{=}\\
s_{n-1} s_{n} \underline{s_{n-2}} s_{n-1}s_{n-2} \cdots s_3t_ks_2 \cdots s_i s_{n-2} \cdots s_{i'} & \stackrel{(1)}{=}\\
\end{array}$\\
$\begin{array}{l}
s_{n-1} s_{n-2} s_{n} s_{n-1}s_{n-2} \cdots s_3t_ks_2 \cdots s_i \underline{s_{n-2}} \cdots \underline{s_{i'+1}}s_{i'}.
\end{array}$\\
We apply the same operations to the underlined letters and\\
we get $s_{n-1} \cdots s_{i'}s_n\cdots s_3t_ks_2 \cdots s_i s_{i'}$.\\
\emph{If $i'=i+1$}, we directly have $s_n(a_{n-1}a_n) \in$ Span$(\Lambda_{n-1}\Lambda_n)$.\\
\emph{If $i' > i+1$}, then we write $s_{i'+1}s_{i'}s_{i'-1}$ in the underlined word of\\ $s_{n-1} \cdots s_{i'}\underline{s_n\cdots s_3}t_ks_2 \cdots s_i s_{i'}$ and get\\
$\begin{array}{ll}
s_{n-1} \cdots s_{i'}s_n \cdots s_{i'+1}s_{i'}s_{i'-1}\cdots s_3t_ks_2 \cdots s_i \underline{s_{i'}} & \stackrel{(1)}{=}\\
s_{n-1} \cdots s_{i'}s_n \cdots s_{i'+1}\underline{s_{i'}s_{i'-1}s_{i'}} \cdots s_3t_ks_2 \cdots s_i  & \stackrel{(2)}{=}\\
s_{n-1} \cdots s_{i'}s_n \cdots s_{i'+1}\underline{s_{i'-1}}s_{i'}s_{i'-1} \cdots s_3t_ks_2 \cdots s_i & \stackrel{(1)}{=}\\
\end{array}$\\
$\begin{array}{l}
s_{n-1} \cdots s_{i'-1}s_n \cdots s_3t_ks_2 \cdots s_i
\end{array}$ which belongs to Span$(\Lambda_{n-1}\Lambda_n)$.

\emph{Suppose $i \geq i'$}. We have $s_n(a_{n-1}a_n)$ is equal to\\
$s_n \cdots s_3t_ks_2 \cdots s_i \underline{s_n} \cdots s_{i'}$. We shift $\underline{s_n}$ to the left and apply a braid relation to get
$s_{n-1} s_{n} s_{n-1} \cdots s_3t_ks_2 \cdots s_i \underline{s_{n-1} \cdots s_{i'}}$. Write $s_{i+2}s_{i+1}$ in the underlined word to get
$s_{n-1} s_{n} s_{n-1} \cdots s_3t_ks_2 \cdots s_i \underline{s_{n-1}} \cdots \underline{s_{i+2}}s_{i+1} \cdots s_{i'}$. We apply the same operations to the underlined letters to get\\
$\begin{array}{ll}
s_{n-1}\cdots s_{i+1}s_n \cdots s_3t_ks_2 \cdots \underline{s_i s_{i+1} s_i} s_{i-1} \cdots s_{i'} & \stackrel{(2)}{=}\\
s_{n-1}\cdots s_{i+1}s_n \cdots s_3t_ks_2 \cdots \underline{s_{i+1}} s_{i} s_{i+1} s_{i-1} \cdots s_{i'} & \stackrel{(1)}{=}\\
\end{array}$\\
$\begin{array}{l}
s_{n-1}\cdots s_{i}s_n \cdots s_3t_ks_2 \cdots s_{i-1}s_{i} s_{i+1}s_{i-1} \cdots s_{i'}
\end{array}$ (The details of the computation is left to the reader). Then we have\\
$\begin{array}{ll}
s_{n-1}\cdots s_{i}s_n \cdots s_3t_ks_2 \cdots s_{i-1}s_{i}s_{i+1}\underline{s_{i-1}} \cdots s_{i'} & \stackrel{(1)}{=}\\
s_{n-1}\cdots s_{i}s_n \cdots s_3t_ks_2 \cdots \underline{s_{i-1}s_{i}s_{i-1}}s_{i+1}s_{i-2} \cdots s_{i'} & \stackrel{(2)}{=}\\
s_{n-1}\cdots s_{i}s_n \cdots s_3t_ks_2 \cdots \underline{s_{i}}s_{i-1}s_{i}s_{i+1}s_{i-2} \cdots s_{i'} & \stackrel{(1)}{=}\\
\end{array}$\\
$\begin{array}{l}
s_{n-1}\cdots s_{i-1}s_n \cdots s_3t_ks_2 \cdots s_{i-2}s_{i-1}s_{i}s_{i+1}\underline{s_{i-2}} \cdots \underline{s_{i'+1}}s_{i'}.
\end{array}$ We apply the same operations to the underlined letters and we finally get\\
$s_{n-1}\cdots s_{i'+1}s_n \cdots s_3t_k\underline{s_2s_3 \cdots s_is_{i+1}}s_{i'}$. We write $s_{i'-1}s_{i'}s_{i'+1}$ in the underlined word and get $s_{n-1}\cdots s_{i'+1}s_n \cdots s_3t_ks_2s_3 \cdots s_{i'-1}s_{i'}s_{i'+1}\cdots s_is_{i+1} \underline{s_{i'}}$. We shift $\underline{s_{i'}}$ to the left. We get $s_{n-1}\cdots s_{i'+1}s_n \cdots s_3t_ks_2s_3 \cdots s_{i'-1}\underline{s_{i'}s_{i'+1}s_{i'}}\cdots s_is_{i+1}$. We apply a braid relation and get $s_{n-1}\cdots s_{i'+1}s_n \cdots s_3t_ks_2s_3 \cdots s_{i'-1} \underline{s_{i'+1}}s_{i'}s_{i'+1}\cdots s_is_{i+1}$. Now we shift $\underline{s_{i'+1}}$ to the left and get\\ 
$\begin{array}{ll}
s_{n-1}\cdots s_{i'+1}s_n \cdots \underline{s_{i'+1}s_{i'}s_{i'+1}}s_{i'-1}\cdots s_3t_ks_2 \cdots s_{i+1} & \stackrel{(2)}{=}\\
s_{n-1}\cdots s_{i'+1}s_n \cdots \underline{s_{i'}}s_{i'+1}s_{i'}s_{i'-1}\cdots s_3t_ks_2 \cdots s_{i+1} & \stackrel{(1)}{=}\\
s_{n-1}\cdots s_{i'+1}s_{i'}s_n \cdots s_{i'+1}s_{i'}s_{i'-1}\cdots s_3t_ks_2 \cdots s_{i+1} & = \\
\end{array}$\\
$\begin{array}{l}
s_{n-1} \cdots s_{i'}s_n \cdots s_3t_ks_2 \cdots s_{i+1},
\end{array}$ which belongs to Span$(\Lambda_{n-1}\Lambda_n)$.

\end{proof}

\begin{lemma}\label{Lemma9}

If $a_{n-1} = s_{n-1} \cdots s_3t_ks_2 \cdots s_i$ with $2 \leq i \leq n-1$, \mbox{$1 \leq k \leq e-1$} and $a_n = s_n s_{n-1} \cdots s_{3}t_l$ with $0 \leq l \leq e-1$, then $s_n (a_{n-1}a_n)$ belongs to Span$(S_{n-1}^{*} \Lambda_n)$.

\end{lemma}

\begin{proof}

By the final result of the computations in Lemma \ref{Lemma8}, we have $s_n(a_{n-1}a_n)$ is equal to
$s_{n-1}\cdots s_3 s_n \cdots s_3t_ks_2s_3 \cdots s_{i+1}\underline{t_l}$. We shift $\underline{t_l}$ to the left and get\\
$s_{n-1}\cdots s_3 s_n \cdots s_3t_ks_2s_3t_l \cdots s_{i+1}$. By Case 5 of Proposition \ref{PropCasen=3}, we have to deal with the following two terms:
\begin{itemize}

\item $s_{n-1}\cdots s_3s_n \cdots s_3 t_x s_3t_ls_4 \cdots s_{i+1}$ and
\item $s_{n-1}\cdots s_3s_n \cdots s_3 t_x s_3 t_l s_3 s_4 \cdots s_{i+1}$ with $1 \leq x,l \leq e-1$.
\end{itemize}
The first term is of the form\\
$\begin{array}{ll}
s_{n-1}\cdots s_3s_n \cdots \underline{s_3 t_x s_3}t_ls_4 \cdots s_{i+1} & \stackrel{(2)}{=}\\
s_{n-1}\cdots s_3s_n \cdots s_4\underline{t_x} s_3 t_x t_ls_4 \cdots s_{i+1} & \stackrel{(1)}{=}\\
s_{n-1}\cdots s_3t_xs_n \cdots s_3 t_x t_l\underline{s_4} \cdots s_{i+1} & \stackrel{(1)}{=}\\
s_{n-1}\cdots s_3t_xs_n \cdots \underline{s_4s_3s_4} t_x t_ls_5 \cdots s_{i+1} & \stackrel{(2)}{=}\\
s_{n-1}\cdots s_3t_xs_n \cdots \underline{s_3}s_4s_3 t_x t_ls_5 \cdots s_{i+1} & \stackrel{(1)}{=}\\
\end{array}$\\
$\begin{array}{l}
s_{n-1}\cdots s_3t_xs_3s_n \cdots s_4s_3 t_x t_l\underline{s_5} \cdots \underline{s_{i+2}}s_{i+1}.
\end{array}$ We apply the same operations to the underlined letters, we get $s_{n-1}\cdots s_3t_xs_3\cdots s_{i-1} s_n \cdots s_4s_3 t_x t_l \underline{s_{i+1}}$. Finally, we shift $\underline{s_{i+1}}$ to the left and get $s_{n-1}\cdots s_3t_xs_3\cdots s_{i} s_n \cdots s_4s_3 t_x t_l$.
Since $2 \leq i \leq n-1$ and by the computation of $t_xt_l$ in Lemma \ref{LemmaTLTKInVect}, the lemma is satisfied for this case.\\
The second term is equal to\\
$\begin{array}{ll}
s_{n-1}\cdots s_3s_n \cdots \underline{s_3 t_x s_3} t_l s_3 s_4 \cdots s_{i+1} & \stackrel{(2)}{=}\\
s_{n-1}\cdots s_3s_n \cdots \underline{t_x} s_3 t_x t_l s_3 s_4 \cdots s_{i+1} & \stackrel{(1)}{=}\\
\end{array}$\\
$\begin{array}{l}
s_{n-1}\cdots s_3t_xs_n \cdots  s_3 t_x t_l s_3 s_4 \cdots s_{i+1}.
\end{array}$ We replace $t_xt_l$ by its value given in Lemma \ref{LemmaTLTKInVect}, we get terms of the three following forms:
\begin{itemize}

\item $s_{n-1}\cdots s_3t_xs_n \cdots  s_3 t_m t_0 s_3 s_4 \cdots s_{i+1}$ with $1 \leq m \leq e-1$,
\item $s_{n-1}\cdots s_3t_xs_n \cdots  s_3 t_m s_3 s_4 \cdots s_{i+1}$ with $0 \leq m \leq e-1$,
\item $s_{n-1}\cdots s_3t_xs_n \cdots  s_4 s_3^2 s_4 \cdots s_{i+1}$.

\end{itemize}
The first term belongs to Span$(S_{n-1}^{*} \Lambda_n)$. The third term also belongs to Span$(S_{n-1}^{*} \Lambda_n)$. This is done by using the computation in the proof of Lemma \ref{LmScholie1}. For the second term, we have\\
$\begin{array}{ll}
s_{n-1}\cdots s_3t_xs_n \cdots s_4\underline{s_3t_ms_3}s_4 \cdots s_{i+1} & \stackrel{(2)}{=}\\
s_{n-1}\cdots s_3t_xs_n \cdots \underline{s_4 t_m}s_3 t_m\underline{s_4} \cdots s_{i+1} & \stackrel{(1)}{=}\\
s_{n-1}\cdots s_3t_xt_ms_n \cdots \underline{s_4s_3s_4}t_m s_5 \cdots s_{i+1} & \stackrel{(2)}{=}\\
s_{n-1}\cdots s_3t_xt_ms_n \cdots \underline{s_3}s_4s_3t_m s_5 \cdots s_{i+1} & \stackrel{(1)}{=}\\
\end{array}$\\
$\begin{array}{l}
s_{n-1}\cdots s_3t_xt_ms_3 s_n \cdots s_4s_3t_m \underline{s_5} \cdots \underline{s_{i+2}} s_{i+1}.
\end{array}$ We apply the same operations to the underlined letters and get $s_{n-1}\cdots s_3t_xt_ms_3\cdots s_{i-1} s_n \cdots s_3t_m s_{i+1}$. Now we shift $\underline{s_{i+1}}$ to the left and finally get $s_{n-1}\cdots s_3t_xt_ms_3\cdots s_{i} \underline{s_n \cdots s_3t_m}$ with $2 \leq i \leq n-1$ which belongs to Span$(S_{n-1}^{*} \Lambda_n)$.\\
Note that for $i = 2$, we get terms that are equal to the itemized terms (given at the beginning of this proof) after replacing $s_4 \cdots s_{i+1}$ by $1$.

\end{proof}

\begin{lemma}\label{Lemma10}

If $a_{n-1} = s_{n-1} \cdots s_3t_ks_2 \cdots s_i$ with $2 \leq i \leq n-1$, \mbox{$1 \leq k \leq e-1$} and $a_n = s_n \cdots s_{3}t_ls_2 \cdots s_{i'}$ with $1 \leq l \leq e-1$ and $2 \leq i' \leq n$, then $s_n (a_{n-1}a_n)$ belongs to Span$(S_{n-1}^{*} \Lambda_n)$.

\end{lemma}

\begin{proof}

According to the computation in the proof of Lemma \ref{Lemma9}, we get the following possible terms. They appear in the proof of Lemma \ref{Lemma9} in the following order.
\begin{enumerate}

\item $s_n \cdots s_3t_xt_l$ with $0 \leq x,l \leq e-1$,
\item $s_n \cdots s_3 t_mt_0 s_3 \cdots s_{i+1}$ with $1 \leq m \leq e-1$,
\item $s_n \cdots s_3 t_m$ with $0 \leq m \leq e-1$,
\item $s_n \cdots s_4s_3^2 s_4 \cdots s_{i+1}$

\end{enumerate}

We show that the product on the right by $s_2 \cdots s_{i'}$ of each of the previous terms belongs to Span$(S_{n-1}^{*}\Lambda_n)$.\\

\emph{Case 1}. Consider the first term $s_n \cdots s_3t_xt_l (s_2 \cdots s_{i'})$ with $0 \leq x,l \leq e-1$.
We replace $t_xt_l$ by its decomposition given in Lemma \ref{LemmaTLTKInVect}, we get these terms
\begin{itemize}

\item $s_n \cdots s_3 t_mt_0 s_3 \cdots s_{i'}$ with $1 \leq m \leq e-1$,
\item $s_n \cdots s_3 t_m s_3 \cdots s_{i'}$ with $0 \leq m \leq e-1$,
\item $s_n \cdots s_4 s_3^2 s_4 \cdots s_{i'}$.

\end{itemize}
The first term belongs to Span$(S_{n-1}^{*} \Lambda_n)$. The third one is done in Lemma \ref{LmScholie1}. For the second one, we have\\
$\begin{array}{ll}
s_n \cdots \underline{s_3t_ms_3} \cdots s_{i'} & \stackrel{(2)}{=}\\
s_n \cdots \underline{t_m}s_3t_m\underline{s_4} \cdots s_{i'} & \stackrel{(1)}{=}\\
t_ms_n \cdots \underline{s_4s_3s_4}t_ms_5 \cdots s_{i'} & \stackrel{(2)}{=}\\
t_ms_n \cdots \underline{s_3}s_4s_3 t_ms_5 \cdots s_{i'} & \stackrel{(1)}{=}\\
\end{array}$\\
$\begin{array}{l}
t_ms_3s_n \cdots s_4s_3 t_m\underline{s_5} \cdots \underline{s_{i'-1}}s_{i'}.
\end{array}$ We apply the same operations to $\underline{s_5}$, $\cdots$, $\underline{s_{i'-1}}$ and get
$t_ms_3 \cdots s_{i'-2}s_n \cdots s_3 t_m \underline{s_{i'}}$. We shift $\underline{s_{i'}}$ to the left and finally get\\
$t_ms_3 \cdots s_{i'-1} \underline{s_n \cdots s_3 t_m}$ which belongs to Span$(S_{n-1}^{*} \Lambda_n)$.\\

We now consider \emph{Case 3} because we use the computation we made in Case 1. In this case, the term is of the form $s_n \cdots s_3t_m(s_2 \cdots s_{i'})$ with $0 \leq m \leq e-1$. If $m \neq 0$, then it belongs to Span$(S_{n-1}^{*} \Lambda_n)$. If $m = 0$, we get two terms $s_n \cdots s_3 t_0 s_3 \cdots s_{i'}$ and $s_n \cdots s_4 s_3^2 s_4 \cdots s_{i'}$. The first term is done in Case 1. The second term is done in Lemma \ref{LmScholie1}.\\

Consider \emph{Case 4}. We replace $s_n \cdots s_4 s_3^2s_4 \cdots s_{i+1}$ by its decomposition given by the computation in the proof of Lemma \ref{LmScholie1}. We multiply each term of the decomposition by $s_2 \cdots s_{i'}$ on the right and we prove that it belongs to Span$(S_{n-1}^{*} \Lambda_n)$ in the same way as the proof of Lemma \ref{Lemma4}.\\

Finally, it remains to show that the term corresponding to \emph{Case 2} belongs to Span$(S_{n-1}^{*} \Lambda_n)$. It is of the form \begin{center}$s_n \cdots s_3 t_mt_0 s_3 \cdots s_{i+1} (s_2 \cdots s_{i'})$ with $1 \leq m \leq e-1$.\end{center}

\emph{Suppose $i' \leq i$}. We have\\
$\begin{array}{ll}
s_n \cdots s_3t_ms_2s_3 \cdots s_{i+1}\underline{s_2} \cdots s_{i'} & \stackrel{(1)}{=}\\
s_n \cdots s_3t_m\underline{s_2s_3s_2}s_4 \cdots s_{i+1}s_3 \cdots s_{i'} & \stackrel{(2)}{=}
\end{array}$\\

$\begin{array}{ll}
s_n \cdots \underline{s_3t_m s_3}s_2s_3 s_4 \cdots s_{i+1}s_3 \cdots s_{i'} & \stackrel{(2)}{=}\\
s_n \cdots \underline{t_m}s_3t_m s_2s_3 s_4 \cdots s_{i+1} s_3 \cdots s_{i'} & \stackrel{(1)}{=}\\
t_m s_n \cdots s_3t_m s_2s_3 s_4 \cdots s_{i+1}\underline{s_3} \cdots s_{i'} & \stackrel{(1)}{=}\\
t_m s_n \cdots s_3t_m s_2\underline{s_3 s_4s_3}s_5 \cdots s_{i+1}s_4 \cdots s_{i'} & \stackrel{(2)}{=}\\
t_m s_n \cdots s_3t_m s_2 \underline{s_4}s_3 s_4s_5 \cdots s_{i+1}s_4 \cdots s_{i'} & \stackrel{(1)}{=}\\
t_m s_n \cdots \underline{s_4s_3s_4}t_m s_2s_3 s_4s_5 \cdots s_{i+1}s_4 \cdots s_{i'} & \stackrel{(2)}{=}\\
t_m s_n \cdots \underline{s_3}s_4s_3t_m s_2s_3 s_4s_5 \cdots s_{i+1}s_4 \cdots s_{i'} & \stackrel{(1)}{=}\\
t_m s_3 s_n \cdots s_4s_3t_m s_2s_3 s_4s_5 \cdots s_{i+1}\underline{s_4} \cdots \underline{s_{i'-1}} s_{i'}.
\end{array}$\\
We apply the same operations to $\underline{s_4}$, $\cdots$, $\underline{s_{i'-1}}$ to get\\
$t_ms_3 \cdots s_{i'-1}s_n \cdots s_3 t_m s_2 s_3 \cdots s_i s_{i+1} \underline{s_{i'}}$. We shift $\underline{s_{i'}}$ to the left and finally get\\
$t_ms_3 \cdots s_{i'}s_n \cdots s_3 t_m s_2 s_3 \cdots s_i s_{i+1}$. Since $i' \leq i \leq n-1$, this term satisfies the property of the lemma.\\

\emph{Suppose $i' > i$}. As previously, we have\\
$\begin{array}{ll}
s_n \cdots s_3t_ms_2s_3 \cdots s_{i+1}\underline{s_2} \cdots s_{i'} & \stackrel{(1)}{=}\\
t_m s_n \cdots s_3t_ms_2s_3 \cdots s_{i+1}\underline{s_3} \cdots s_{i'} & \stackrel{(1)}{=}\\
t_ms_3s_n \cdots s_3t_ms_2s_3 \cdots s_{i+1}s_4 \cdots s_{i'}. & \\
\end{array}$\\ Now we write $s_is_{i+1}$ in $s_4 \cdots s_{i'}$ and get $t_ms_3s_n \cdots s_3t_ms_2s_3 \cdots s_{i+1}\underline{s_4} \cdots \underline{s_i}s_{i+1} \cdots s_{i'}$.\\
We apply the same operations to $\underline{s_4}$, $\cdots$, $\underline{s_i}$ to get\\
$t_ms_3 \cdots s_i s_n \cdots s_3t_ms_2s_3 \cdots s_{i}s_{i+1}^2s_{i+2} \cdots s_{i'}$. Applying a quadratic relation, we finally get
$a t_ms_3 \cdots s_i s_n \cdots s_3t_m s_2 s_3 \cdots s_{i'}
+ t_m s_3 \cdots s_i s_n \cdots s_3t_ms_2s_3 \cdots s_{i}s_{i+2} \cdots s_{i'}$.\\
The first term satisfies the property of the lemma.\\
For the second term, we write $s_{i+2}s_{i+1}$ in $\underline{s_n \cdots s_3}$ of \\
$t_m s_3 \cdots s_i \underline{s_n \cdots s_3}t_ms_2s_3 \cdots s_{i} s_{i+2} \cdots s_{i'}$ and get\\
$\begin{array}{ll}
t_m s_3 \cdots s_i s_n \cdots s_{i+2}s_{i+1} \cdots s_3t_ms_2s_3 \cdots s_{i} \underline{s_{i+2}} \cdots s_{i'} & \stackrel{(1)}{=}\\
t_m s_3 \cdots s_i s_n \cdots \underline{s_{i+2}s_{i+1}s_{i+2}} \cdots s_3t_ms_2s_3 \cdots s_{i} s_{i+3} \cdots s_{i'} & \stackrel{(2)}{=}\\
t_m s_3 \cdots s_i s_n \cdots \underline{s_{i+1}}s_{i+2}s_{i+1} \cdots s_3t_ms_2s_3 \cdots s_{i} s_{i+3} \cdots s_{i'} & \stackrel{(1)}{=}\\
\end{array}$\\
$\begin{array}{l}
t_m s_3 \cdots s_{i+1} s_n \cdots s_3t_ms_2s_3 \cdots s_{i} \underline{s_{i+3}} \cdots \underline{s_{i'}}.
\end{array}$\\
We apply the same operations to $\underline{s_{i+3}}$, $\cdots$, $\underline{s_{i'}}$ and finally get\\
\begin{small}$t_m s_3 \cdots s_{i'-1} s_n \cdots s_3t_ms_2s_3 \cdots s_{i}$.\end{small} Since \begin{small}$i'-1 \leq n-1$\end{small}, this term belongs to \begin{small}Span$(S_{n-1}^{*} \Lambda_n)$.\end{small}

\end{proof}

\bigskip

As a consequence of Theorem \ref{TheoremNewBasis}, by the second item of Proposition 2.3 of \cite{MarinG20G21}, we get another proof of the BMR freeness conjecture (see Theorem \ref{PropositionBMRFreenessTheorem}) for the complex reflection groups $G(e,e,n)$.

\begin{remark}

Our basis never coincides with the Ariki basis \cite{ArikiHecke} for the Hecke algebra associated with $G(e,e,n)$. For example, consider the element $t_1t_0.t_0$ which belongs to Ariki's basis. In our basis, it is equal to the linear combination $a t_1t_0 + t_1$, where $t_1t_0$ and $t_1$ are two distinct elements of our basis.

\end{remark}

The general hope would be to construct natural bases for $H(de,e,n)$ that can be defined from geodesic normal forms in the associated complex reflection groups $G(de,e,n)$, which was established in Theorem \ref{TheoremNewBasis} for the case of $H(e,e,n)$. This arises the question whether the geodesic normal forms established in Section \ref{SectionReducedWordsG(de,e,n)} by using the presentation of Corran-Lee-Lee of $G(de,e,n)$ for $d > 1$ also provide natural bases for the corresponding Hecke algebras. Unfortunately, some intricate arguments in our proof for the case of $H(e,e,n)$ does not work in the general case of $H(de,e,n)$. Actually, on the one hand, the Hecke algebra is defined, in the general case, as a quotient of the complex braid group algebra defined by the presentation of Corran-Lee-Lee, where the generators $t_i$'s are defined over $\mathbb{Z}$, but on the other hand our geodesic normal forms are attached to the presentation of Corran-Lee-Lee of the complex reflection group $G(de,e,n)$, where the generators $\mathbf{t}_i$'s are defined this time over $\mathbb{Z}/de\mathbb{Z}$. This phenomenon does not occur in the case of the presentations of Corran-Picantin of $G(e,e,n)$ and $B(e,e,n)$ used to construct a basis for $H(e,e,n)$. Nonetheless, for $e=1$, the Hecke algebra $H(d,1,n)$ is defined by using the classical presentation of the complex braid group $B(d,1,n)$. Using the geodesic normal forms constructed in Subsection \ref{SubGeodesicG(d,1,n)}, we are also able to provide a natural basis for $H(d,1,n)$. This will be established in the next section.

\section{The case of \texorpdfstring{$H(d,1,n)$}{TEXT}}\label{SectionTheCaseOfG(d,1,n)}

Let $d > 1$ and $n \geq 2$. Let $R_0 = \mathbb{Z}[a,b_1,b_2, \cdots, b_{d-1}]$. Recall that the Hecke algebra $H(d,1,n)$ is defined as the unitary associative $R_0$-algebra generated by the elements $z, s_2, s_3, \cdots, s_n$ with the following relations:
\begin{enumerate}

\item $zs_2zs_2 = s_2zs_2z$,
\item $zs_j = s_jz$ for $2 \leq j \leq n$,
\item $s_is_{i+1}s_i = s_{i+1}s_is_{i+1}$ for $2 \leq i \leq n-1$,
\item $s_is_j = s_js_i$ for $|i-j| > 1$,
\item $z^d -b_1 z^{d-1} - b_2 z^{d-2} - \cdots - b_{d-1} z -1 = 0$ and $s_j^2 - as_j - 1 = 0$ for $2 \leq j \leq n$.

\end{enumerate}

Using the geodesic normal forms introduced in subsection \ref{SubGeodesicG(d,1,n)} for all the elements of $G(d,1,n)$, we construct a basis for $H(d,1,n)$ that is different from the one defined by Ariki-Koike in \cite{ArikiKoikeHecke}.\\

Let us introduce the following subsets of $H(d,1,n)$.
\begin{center}
\begin{tabular}{lll}
$\Lambda_1$ & $=$ & $\{z^k$ for $0 \leq k \leq d-1\}$,\\
\end{tabular}
\end{center}
and for $2 \leq i \leq n$,
\begin{center}
\begin{tabular}{llll}
$\Lambda_i$ & $=$ & $\{1$,\\
 & & $\ \ s_i \cdots s_{i'}$ & for $2 \leq i' \leq i$,\\ 
 & & $\ \ s_i \cdots s_{2}z^k$ & for $1 \leq k \leq d-1$,\\
 & & $\ \ s_i \cdots s_{2}z^ks_2 \cdots s_{i'}$ & for $1 \leq k \leq d-1$ and $2 \leq i' \leq i\}$.\\
\end{tabular}
\end{center}

\noindent Define $\Lambda = \Lambda_1\Lambda_2 \cdots \Lambda_n$ to be the set of the products $a_1a_2 \cdots a_n$, where $a_1 \in \Lambda_1, \cdots,\\ a_n \in \Lambda_n$. Remark that this set corresponds to all the reduced words $R\!E(w)$ of the form $R\!E_1(w)R\!E_2(w) R\!E_3(w) \cdots R\!E_n(w)$ introduced in subsection \ref{SubGeodesicG(d,1,n)} (see Definition \ref{DefinitionREG(d,1,n)}). In this section, we establish the following theorem.

\begin{theorem}\label{TheoremNewBasisH(d,1,n)}

The set $\Lambda$ provides an $R_0$-basis of the Hecke algebra $H(d,1,n)$.

\end{theorem}

We have $|\Lambda_1| = d$ and $|\Lambda_i| = id$ for $2 \leq i \leq n$. Then $|\Lambda|$ is equal to $d^n n!$ that is the order of $G(d,1,n)$. Hence by Proposition 2.3$(i)$ of \cite{MarinG20G21}, it is sufficient to show that $\Lambda$ is an $R_0$-generating set of $H(d,1,n)$. This is proved by induction on $n$ in much the same way as Theorem \ref{TheoremNewBasis}. We provide the following preliminary lemmas that are useful in the proof of the theorem.

\begin{lemma}\label{Lemma(s2ts2)^k}

For $1 \leq k \leq d-1$, the element $(s_2zs_2)^k$ belongs to $\sum\limits_{\substack{\lambda_1 \in \Lambda_1,\\ \lambda_2 \in \Lambda'_2}}R_0 \lambda_1 \lambda_2$, where $\Lambda'_2 = \{1, s_2, s_2z, s_2z^2, \cdots, s_2z^{k-1}, s_2zs_2, s_2z^2s_2, \cdots, s_2z^ks_2 \}$.

\end{lemma}

\begin{proof}

The property is clear for $k=1$. Let $k=2$. We have
$(s_2zs_2)^2 = s_2zs_2^2zs_2$. We apply a quadratic relation and get
$a s_2zs_2zs_2 + s_2z^2s_2$. Applying a braid relation, one gets
$a zs_2zs_2^2 + s_2z^2s_2$. Using a quadratic relation, this is equal to
$a^2 zs_2zs_2 + a zs_2z + s_2z^2s_2$, where each term is of the form $\lambda_1\lambda_2$ with $\lambda_1 \in \Lambda_1$ and $\lambda_2 \in \{1,s_2,s_2z,s_2zs_2,s_2z^2s_2\}$.

Let $k \geq 3$. Suppose the property is satisfied for $(s_2zs_2)^3$, $\cdots$, and $(s_2zs_2)^{k-1}$. We have $(s_2zs_2)^k = (s_2zs_2)^{k-1}(s_2zs_2)$. By the induction hypothesis, the terms that appear in the decomposition of $(s_2zs_2)^{k-1}$ are of the following forms.
\begin{itemize}

\item $z^c$ with $0 \leq c \leq d-1$,
\item $z^cs_2z^{c'}$ with $0 \leq c \leq d-1$ and $0 \leq c' \leq k-2$,
\item $z^cs_2z^{c'}s_2$ with $0 \leq c \leq d-1$ and $1 \leq c' \leq k-1$.

\end{itemize}
Multiplying these terms by $s_2zs_2$ on the right, we get the following $3$ cases.

\emph{Case 1}: A term of the form $z^c s_2zs_2$ with $0 \leq c \leq d-1$. It is of the form $\lambda_1\lambda_2$ with $\lambda_1 \in \Lambda_1$ and $\lambda_2 \in \Lambda'_2$.

\emph{Case 2}: A term of the form $z^cs_2z^{c'}s_2zs_2$ with $0 \leq c \leq d-1$ and $0 \leq c' \leq k-2$. We shift $z^{c'}$ to the right by applying braid relations and get
$z^cs_2^2zs_2z^{c'}$. Applying a quadratic relation, this is equal to
$a z^cs_2zs_2\underline{z^{c'}} +  z^{c+1}s_2z^{c'}$. Now we shift $\underline{z^{c'}}$ to the left by applying braid relations and get
$a z^{c+c'}s_2zs_2 +  z^{c+1}s_2z^{c'}$. Each term is of the form $\lambda_1\lambda_2$ with $\lambda_1 \in \Lambda_1$ and $\lambda_2 \in \Lambda'_2$.

\emph{Case 3}: A term of the form $z^cs_2z^{c'}s_2^2zs_2$ with $0 \leq c \leq d-1$ and $1 \leq c' \leq k-1$. By applying a quadratic relation, we have $z^cs_2z^{c'}s_2^2zs_2 = a z^cs_2z^{c'}s_2zs_2 + z^cs_2z^{c'+1}s_2$. The first term is the same as in the previous case. Then both terms are of the form $\lambda_1\lambda_2$ with $\lambda_1 \in \Lambda_1$ and $\lambda_2 \in \Lambda'_2$.

\end{proof}

\begin{lemma}\label{Lemma(s2ts2)^ks2}

For $1 \leq k \leq d-1$, the element $(s_2zs_2)^ks_2$ belongs to $R_0(s_2zs_2)^k + R_0z(s_2zs_2)^{k-1} + \cdots + R_0z^{k-1}(s_2zs_2) + R_0s_2z^k$.

\end{lemma}

\begin{proof}

For $k = 1$, we have $(s_2zs_2)s_2 = s_2zs_2^2 = a s_2zs_2 + s_2z$. Then the property is satisfied for $k = 1$. Let $k \geq 2$. Suppose that the property is satisfied for $(s_2zs_2)^{k-1}$. We have $(s_2zs_2)^k s_2 = (s_2zs_2)(s_2zs_2)^{k-1} s_2$. By the induction hypothesis, it belongs to $R_0(s_2zs_2)(s_2zs_2)^{k-1} + R_0(s_2zs_2)z(s_2zs_2)^{k-2} + \cdots + R_0(s_2zs_2)z^{k-2}(s_2zs_2) + R_0(s_2zs_2)s_2z^{k-1}$. Then it belongs to $R_0(s_2zs_2)^{k} + R_0z(s_2zs_2)^{k-1} + \cdots + R_0z^{k-2}(s_2zs_2)^2\\ + R_0z^{k-1}(s_2zs_2) + R_0s_2z^k$. It follows that for all $1 \leq k \leq d-1$, the element $(s_2zs_2)^ks_2$ belongs to $R_0(s_2zs_2)^k + R_0z(s_2zs_2)^{k-1} + \cdots + R_0z^{k-1}(s_2zs_2) + R_0s_2z^k$.

\end{proof}

\begin{lemma}\label{Lemmas2t^ks2In(s2ts2)^k}

For $1 \leq k \leq d-1$, the element $s_2z^ks_2$ belongs to $\sum\limits_{\substack{\lambda_1 \in \Lambda_1,\\ \lambda_2 \in \Lambda''_2}}R_0 \lambda_1 \lambda_2$, where $\Lambda''_2 = \{1,s_2,s_2z,s_2z^2, \cdots, s_2z^{k-1}, s_2zs_2, (s_2zs_2)^2, \cdots, (s_2zs_2)^{k} \}$.

\end{lemma}

\begin{proof}

The lemma is satisfied for $k=1$. For $k=2$, we have
$s_2z^2s_2 = s_2zs_2^{-1}s_2zs_2$. Using that $s_2^{-1} = s_2 - a$, we get
$s_2zs_2^2zs_2 - a \underline{s_2zs_2z}s_2$. Now we apply a braid relation then a quadratic relation and get
$(s_2zs_2)^2 - a zs_2zs_2^2 = (s_2zs_2)^2 - a^2 zs_2zs_2 - a zs_2z$ which satisfies the property we are proving.

Suppose the property is satisfied for $s_2z^{k-1}s_2$. We have
$s_2z^ks_2 = s_2z^{k-1}s_2^{-1}s_2zs_2 = s_2z^{k-1}s_2s_2zs_2 - a s_2z^{k-1}s_2zs_2$ by replacing $s_2^{-1}$ by $s_2 - a$.
For the second term, we shift $z^{k-1}$ to the right and get
$s_2^2zs_2z^{k-1}$. We apply a quadratic relation to get
$a s_2zs_2\underline{z^{k-1}} + zs_2z^{k-1}$ then we shift $\underline{z^{k-1}}$ to the left and finally get
$a z^{k-1}s_2zs_2 + zs_2z^{k-1}$, where each term is of the form $\lambda_1 \lambda_2$ with $\lambda_1 \in \Lambda_1$ and $\lambda_2 \in \Lambda''_2$.\\
For the first term $s_2z^{k-1}s_2s_2zs_2$, by the induction hypothesis, the terms that appear in the decomposition of $s_2z^{k-1}s_2$ are of the following forms.
\begin{itemize}

\item $z^c$ and $z^cs_2$ with $0 \leq c \leq d-1$,
\item $z^c(s_2zs_2)^{c'}$ with $0 \leq c \leq d-1$ and $1 \leq c' \leq k-1$,
\item $z^cs_2z^{c'}$ with $0 \leq c \leq d-1$ and $1 \leq c' \leq k-2$.

\end{itemize}
Multiplying these terms by $s_2zs_2$ on the right, we get the following $3$ cases.

\emph{Case 1}. We have $z^c s_2zs_2$ and $z^c\underline{s_2s_2}zs_2 = a z^cs_2zs_2 + z^{c+1}s_2$, where each term in both expressions is of the form $\lambda_1 \lambda_2$ with $\lambda_1 \in \Lambda_1$ and $\lambda_2 \in \Lambda''_2$.

\emph{Case 2}. The term $z^c(s_2zs_2)^{c'}s_2zs_2 = z^c(s_2zs_2)^{c'+1}$ is of the form $\lambda_1 \lambda_2$ with $\lambda_1 \in \Lambda_1$ and $\lambda_2 \in \Lambda''_2$ since $1 \leq c' \leq k-1$.

\emph{Case 3}. We have $z^cs_2z^{c'}s_2zs_2 =
z^cs_2^2zs_2z^{c'} = a z^cs_2zs_2z^{c'} + z^{c+1}s_2z^{c'}$. The first term is equal to $z^{c+c'}s_2zs_2$ and the second term is equal to $z^{c+1}s_2z^{c'}$ with $1 \leq c' \leq k-2$. Both are of the form $\lambda_1 \lambda_2$ with $\lambda_1 \in \Lambda_1$ and $\lambda_2 \in \Lambda''_2$.

\end{proof}

The following proposition ensures that the case $n=2$ of Theorem \ref{TheoremNewBasisH(d,1,n)} works properly.

\begin{proposition}\label{PropInductionHypothesisG(d,1,n)}

For all $a_1 \in \Lambda_1$ and $a_2 \in \Lambda_2$, the elements $za_1a_2$ and $s_2 a_1a_2$ belong to \emph{Span}$(\Lambda_1\Lambda_2)$.

\end{proposition}

\begin{proof}

It is readily checked that $za_1a_2$ belongs to Span$(\Lambda_1\Lambda_2)$. Note that when the power of $z$ exceeds $d-1$, we use Relation 1 of Definition \ref{DefH(de,e,n)}.

It is easily checked that if $a_1 \in \Lambda_1$ and $a_2 = 1$, the element $s_2 a_1a_2$ belongs to Span$(\Lambda_1\Lambda_2)$. Also, when $a_1 = 1$ and $a_2 \in \Lambda_2$, we have that $s_2 a_1a_2$ belongs to Span$(\Lambda_1\Lambda_2)$.

Suppose $a_1 = z^k$ with $1 \leq k \leq d-1$ and $a_2 = s_2$. We have $s_2 a_1 a_2$ is equal to $s_2 z^k s_2$. Hence it belongs to Span$(\Lambda_1\Lambda_2)$.

Suppose $a_1 = z^k$ with $1 \leq k \leq d-1$ and $a_2 = s_2 z^{k'}$ with $1 \leq k' \leq d-1$. We have $s_2 a_1 a_2 = s_2 z^k s_2 z^{k'}$. We replace $s_2 z^k s_2$ by its decomposition given in Lemma \ref{Lemmas2t^ks2In(s2ts2)^k}, then we use the result of Lemma \ref{Lemma(s2ts2)^k} to directly deduce that $s_2 z^k s_2 z^{k'}$ belongs to Span$(\Lambda_1\Lambda_2)$.

Finally, suppose $a_1 = z^k$ with $1 \leq k \leq d-1$ and $a_2 = s_2 z^{k'}s_2$ with $1 \leq k' \leq d-1$.\\
We have $s_2 a_1 a_2$ is equal to $s_2 z^k s_2 z^{k'} s_2$. We replace $s_2 z^k s_2$ by its decomposition given in Lemma \ref{Lemmas2t^ks2In(s2ts2)^k}. Then by the results of Lemmas \ref{Lemma(s2ts2)^ks2} and \ref{Lemma(s2ts2)^k}, we deduce that $s_2 z^k s_2 z^{k'} s_2$ belongs to Span$(\Lambda_1\Lambda_2)$.

\end{proof}

In order to prove Theorem \ref{TheoremNewBasisH(d,1,n)}, we introduce Lemmas \ref{LmScholie11} to \ref{Lemma1010} below that are similar to Lemmas \ref{LmScholie1} to \ref{Lemma10}. Along with Proposition \ref{PropInductionHypothesisG(d,1,n)}, they provide an inductive proof of Theorem \ref{TheoremNewBasisH(d,1,n)} that is similar to the proof of Theorem \ref{TheoremNewBasis}. Therefore, by Proposition 2.3$(ii)$ of \cite{MarinG20G21}, we get a new proof of Theorem \ref{PropositionBMRFreenessTheorem} for the case of the complex reflection groups $G(d,1,n)$. The remaining part of this section is devoted to the proof of Lemmas \ref{LmScholie11} to \ref{Lemma1010}. Let $n \geq 3$. Denote by $S_{n-1}^{*}$ the set of the words over $\{z,s_2,s_3, \cdots, s_{n-1}\}$.

\begin{lemma}\label{LmScholie11}

Let $2 \leq i \leq n$. We have $s_n \cdots s_3 s_2^2s_3 \cdots s_i$ belongs to \emph{Span}$(S_{n-1}^{*}\Lambda_n)$.

\end{lemma}

\begin{proof}

The proof is the same as the proof of Lemma \ref{LmScholie1}. If $i = 2$, we have that $s_n \cdots s_3 s_2^2$ belongs to $R_0s_n \cdots s_3s_2 + R_0 s_n \cdots s_3$. We continue as in the proof of Lemma \ref{LmScholie1}, we get two terms of the form (see line 7 of the proof of Lemma \ref{LmScholie1}):\\ $s_n \cdots s_{k+1}s_ks_{k+1} \cdots s_i$ with $k+1 \leq i$ and $s_n \cdots s_{i+1}s_is_{i-1}^2s_i$.

We also get that the first term $s_n \cdots s_{i+1}s_is_{i-1}^2s_i$ belongs to $R_0 s_{i-1}s_n \cdots s_{i-1} + R_0 s_n \cdots s_{i+1}s_i + R_0 s_n \cdots s_{i+1}$ (the last term is equal to $1$ if $i = n$). Hence the element $s_n \cdots s_{i+1}s_is_{i-1}^2s_i$ belongs to Span$(S_{n-1}^{*}\Lambda_n)$. 

As in the proof of Lemma \ref{LmScholie1}, the element $s_n \cdots s_{k+1}s_ks_{k+1} \cdots s_i$ is equal to $s_ks_{k+1} \cdots s_{i-1}s_n s_{n-1} \cdots s_k$. Hence it belongs to Span$(S_{n-1}^{*}\Lambda_n)$.

\end{proof}

\begin{lemma}\label{Lemma22}

If $a_{n-1} = s_{n-1}s_{n-2} \cdots s_i$ with $2 \leq i \leq n-1$ and $a_n = s_n \cdots s_{i'}$ with $2 \leq i' \leq n$, then $s_n(a_{n-1}a_n)$ belongs to \emph{Span}$(\Lambda_{n-1}\Lambda_n)$.

\end{lemma}

\begin{proof}

The proof is the same as for Lemma \ref{Lemma2}.

If $i < i'$, then we get $s_n(a_{n-1}a_n) = s_{n-1}s_{n-2} \cdots s_{i'-1}s_n \cdots s_i$ which belongs to Span$(\Lambda_{n-1}\Lambda_n)$.

If $i \geq i'$, then we get $s_n(a_{n-1}a_n) = a s_{n-1}\cdots s_i s_n \cdots s_{i'} + s_n \cdots s_{i'} s_n \cdots s_{i+1}$ which also belongs to Span$(\Lambda_{n-1}\Lambda_n)$.

\end{proof}

\begin{lemma}\label{Lemma33}

If $a_{n-1} = s_{n-1}\cdots s_{i}$ with $2 \leq i \leq n-1$ and $a_n = s_n \cdots s_2 z^k$ with $0 \leq k \leq d-1$, then $s_n(a_{n-1}a_n)$ belongs to \emph{Span}$(\Lambda_{n-1}\Lambda_n)$.

\end{lemma}

\begin{proof}

This case correspond to the case $i'=2$ in the proof of Lemma \ref{Lemma22} with a right multiplication by $z^k$ for $0 \leq k \leq d-1$. Since $i \geq 2$, by the case $i \geq i'$ in the proof of Lemma \ref{Lemma22}, we get, using the same technique as in the proof of Lemma \ref{Lemma3}, that $s_n(a_{n-1}a_n)$ is equal to $a s_{n-1} \cdots s_i s_n \cdots s_2 z^k + s_{n-1}\cdots s_2 z^k s_n \cdots s_{i+1}$ which belongs to Span$(\Lambda_{n-1}\Lambda_n)$.

\end{proof}

\begin{lemma}\label{Lemma44}

If $a_{n-1} = s_{n-1} \cdots s_i$ with $2 \leq i \leq n-1$ and $a_n = s_n \cdots s_2 z^k s_2 \cdots s_{i'}$ with $1 \leq k \leq d-1$ and $2 \leq i' \leq n$, then $s_n(a_{n-1}a_n)$ belongs to \emph{Span}$(\Lambda_{n-1}\Lambda_n)$.

\end{lemma}

\begin{proof}

The proof is exactly the same as the proof of Lemma \ref{Lemma4}. According to the computation in the proof of Lemma \ref{Lemma33}, we have\\
$s_n(a_{n-1}a_n) = a s_{n-1} \cdots s_i s_n \cdots s_2 z^k (s_2 \cdots s_{i'}) + s_{n-1} \cdots s_2 z^k s_n \cdots s_{i+1}(s_2 \cdots s_{i'})$.\\
The first term is an element of Span$(\Lambda_{n-1}\Lambda_n)$. For the second term, we follow exactly the proof (starting from line 3) of Lemma \ref{Lemma4} :\\
If $i' < i$, we get $s_{n-1} \cdots s_2 z^k s_2 \cdots s_{i'} s_n \cdots s_{i+1}$ which belongs to Span$(\Lambda_{n-1}\Lambda_n)$.\\
If $i' \geq i$, we get $s_{n-1} \cdots s_2 z^k s_2 \cdots s_{i'-1} s_n \cdots s_i$ which also belongs to Span$(\Lambda_{n-1}\Lambda_n)$.

\end{proof}

\begin{lemma}\label{Lemma55}

If $a_{n-1} = s_{n-1} \cdots s_2 z^k$ with $1 \leq k \leq d-1$ and $a_n = s_n \cdots s_i$ with $2 \leq i \leq n$, then $s_n(a_{n-1}a_n)$ belongs to \emph{Span}$(\Lambda_{n-1}\Lambda_n)$.

\end{lemma}

\begin{proof}

We follow exactly the proof of Lemma \ref{Lemma5} and we get the following terms.

If $i = 2$, we get $s_{n-1} \cdots s_2 s_n \cdots s_2 z^k s_2$ (see line 9 of the proof of Lemma \ref{Lemma5}) which belongs to Span$(\Lambda_{n-1}\Lambda_n)$.

If $i > 2$, we get $s_{n-1} \cdots s_{i-1} s_n \cdots s_2 z^k$ (see the last line of the proof of Lemma \ref{Lemma5}) which also belongs to Span$(\Lambda_{n-1}\Lambda_n)$.

\end{proof}

\begin{lemma}\label{Lemma66}

If $a_{n-1} = s_{n-1} \cdots s_2z^k$ with $1 \leq k \leq d-1$ and $a_n = s_n \cdots s_2 z^l$ with $1 \leq l \leq d-1$, then $s_n(a_{n-1}a_n)$ belongs to \emph{Span}$(\Lambda_{n-1}\Lambda_n)$.

\end{lemma}

\begin{proof}

By the case $i = 2$ in the proof of Lemma \ref{Lemma55}, we get\\ $s_n(a_{n-1}a_n) = s_{n-1} \cdots s_2 s_n \cdots s_3s_2z^ks_2z^l$. If we replace $s_2z^ks_2z^l$ by its decomposition over $\Lambda_1\Lambda_2$ (this is the case $n=2$ of Theorem \ref{TheoremNewBasisH(d,1,n)}, see Proposition \ref{PropInductionHypothesisG(d,1,n)}), we get the three following terms:
\begin{itemize}
\item $s_{n-1} \cdots s_2 s_n \cdots s_3 z^c$ with $0 \leq c \leq d-1$,
\item $s_{n-1} \cdots s_2 s_n \cdots s_3 z^cs_2 z^{c'}$ with $0 \leq c \leq d-1$ and $0 \leq c' \leq d-1$,
\item $s_{n-1} \cdots s_2 s_n \cdots s_3 z^cs_2 z^{c'}s_2$ with $0 \leq c \leq d-1$ and $1 \leq c' \leq d-1$.
\end{itemize}
The first term is equal to $s_{n-1} \cdots s_2 z^c s_n \cdots s_3$ which belongs to Span$(\Lambda_{n-1}\Lambda_n)$.\\
The second term is equal to $s_{n-1} \cdots s_2 z^c s_n \cdots s_3s_2z^{c'}$ which belongs to Span$(\Lambda_{n-1}\Lambda_n)$.\\
Finally, the third term is equal to $s_{n-1} \cdots s_2 z^c s_n \cdots s_2 z^{c'}s_2$ which also belongs to Span$(\Lambda_{n-1}\Lambda_n)$.

\end{proof}

\begin{lemma}\label{Lemma77}

If $a_{n-1} = s_{n-1} \cdots s_2 z^k$ with $1 \leq k \leq d-1$ and $a_n = s_n \cdots s_2 z^l s_2 \cdots s_i$ with $2 \leq i \leq n$ and $1 \leq l \leq d-1$, then $s_n(a_{n-1}a_n)$ belongs to \emph{Span}$(S_{n-1}^{*}\Lambda_n)$.

\end{lemma}

\begin{proof}

According to the proof of the previous lemma, we have to deal with the following three terms:
\begin{itemize}
\item $s_{n-1} \cdots s_2 z^c s_n \cdots s_3 (s_2 \cdots s_i)$ for $0 \leq c \leq d-1$,
\item $s_{n-1} \cdots s_2 z^c s_n \cdots s_3s_2z^{c'}(s_2 \cdots s_i)$ for $0 \leq c \leq d-1$ and $0 \leq c' \leq d-1$,
\item $s_{n-1} \cdots s_2 z^c s_n \cdots s_2 z^{c'}s_2(s_2 \cdots s_{i})$ for $0 \leq c \leq d-1$ and $1 \leq c' \leq d-1$.
\end{itemize}

For the first case, we have $s_n \cdots \underline{s_3s_2s_3}s_4 \cdots s_i = s_n \cdots \underline{s_2} s_3 s_2 \underline{s_4} \cdots s_i$. We shift $\underline{s_2}$ and $\underline{s_4}$ to the left in the previous expression and get $s_2s_n \cdots \underline{s_4s_3s_4}s_2 s_5 \cdots s_i = s_2s_n \cdots \underline{s_3}s_4s_3s_2s_5 \cdots s_i$. We shift $\underline{s_3}$ to the left in the previous expression and get $s_2s_3 s_n \cdots s_2 \underline{s_5} \cdots \underline{s_i}$. We apply the same operations to $\underline{s_5}$, $\cdots$, $\underline{s_i}$ and get\\ $s_2 s_3 \cdots s_{i-1}s_n \cdots s_2$. Hence the term of the first case can be written as follows $s_{n-1} \cdots s_2 z^c s_2 \cdots s_{i-1}s_n \cdots s_2$. Since $i-1 \leq n-1$, it belongs to Span$(S_{n-1}^{*}\Lambda_n)$.

For the second case, we have a term of the form $s_{n-1} \cdots s_2 z^c s_n \cdots s_3s_2z^{c'}(s_2 \cdots s_i)$ for $0 \leq c \leq d-1$ and $0 \leq c' \leq d-1$. If $c' \neq 0$, this term belongs to Span$(S_{n-1}^{*}\Lambda_n)$ and if $c' = 0$, by the computation in the proof of Lemma \ref{LmScholie11}, it also belongs to Span$(S_{n-1}^{*}\Lambda_n)$.

For the third case, we have $s_{n-1} \cdots s_2 z^c s_n \cdots s_2 z^{c'}s_2^2s_3 \cdots s_{i} =$\\
$a s_{n-1} \cdots s_2 z^c s_n \cdots s_2 z^{c'}s_2 \cdots s_{i} + s_{n-1} \cdots s_2 z^c s_n \cdots s_2 z^{c'}s_3 \cdots s_{i}$. The first term is an element of Span$(S_{n-1}^{*}\Lambda_n)$. For the second term, we have $s_{n-1} \cdots s_2 z^c s_n \cdots s_2 \underline{z^{c'}} s_3 \cdots s_{i}\\ = s_{n-1} \cdots s_2 z^c s_n \cdots s_2 s_3 \cdots s_{i} z^{c'}$, where $ s_n \cdots s_2 s_3 \cdots s_{i}$ is already computed in the first case. It is equal to $s_2 \cdots s_{i-1}s_n \cdots s_2$. Hence the term is of the form\\ $s_{n-1} \cdots s_2 z^c s_2 \cdots s_{i-1} s_n \cdots s_2 z^{c'}$ which belongs to Span$(S_{n-1}^{*}\Lambda_n)$.

\end{proof}

\begin{lemma}\label{Lemma88}

If $a_{n-1} = s_{n-1} \cdots s_2 z^k s_2 \cdots s_i$ with $2 \leq i \leq n-1$, $1 \leq k \leq d-1$ and $a_n = s_n \cdots s_{i'}$ with $2 \leq i' \leq n$, then $s_n(a_{n-1}a_n)$ belongs to \emph{Span}$(\Lambda_{n-1}\Lambda_n)$.

\end{lemma}

\begin{proof}

We follow exactly the proof of Lemma \ref{Lemma8}. 

\textit{Suppose $i < i'$}. If $i' = i+1$, we get $s_n(a_{n-1}a_n) = s_{n-1} \cdots s_{i+1}s_n \cdots s_2 z^k s_2 \cdots s_{i+1}$ which belongs to Span$(\Lambda_{n-1}\Lambda_n)$, see lines 9 and 10 of the proof of Lemma \ref{Lemma8}. If $i' > i+1$, we get $s_{n-1} \cdots s_{i'-1}s_n \cdots s_2 z^k s_2 \cdots s_i$ which belongs to Span$(\Lambda_{n-1}\Lambda_n)$, see line 16 of the proof of Lemma \ref{Lemma8}.

\textit{Suppose $i \geq i'$}. We get $s_n(a_{n-1}a_n) = s_{n-1} \cdots s_{i'}s_n \cdots s_2z^ks_2 \cdots s_{i+1}$, see the last line of the proof of Lemma \ref{Lemma8}. Hence $s_n(a_{n-1}a_n)$ belongs to Span$(\Lambda_{n-1}\Lambda_n)$.

\end{proof}

\begin{lemma}\label{Lemma99}

If $a_{n-1} = s_{n-1} \cdots s_2z^ks_2 \cdots s_i$ with $2 \leq i \leq n-1$, $1 \leq k \leq d-1$ and $a_n = s_n \cdots s_2 z^l$ with $0 \leq l \leq d-1$, then $s_n(a_{n-1}a_n)$ belongs to \emph{Span}$(S_{n-1}^{*}\Lambda_n)$.

\end{lemma}

\begin{proof}

According to the computations in the proof of the previous lemma, we have $s_n(a_{n-1}a_n) = s_{n-1} \cdots s_2 s_n \cdots s_2 z^k s_2 \cdots s_{i+1} z^l$. We shift $z^l$ to the left and get\\ $s_n(a_{n-1}a_n) = s_{n-1} \cdots s_2 s_n \cdots s_3(s_2 z^k s_2)z^l s_3 \cdots s_{i+1}$. If we replace $(s_2z^ks_2)z^l$ by its decomposition over $\Lambda_1\Lambda_2$ (this is the case $n=2$ of Theorem \ref{TheoremNewBasisH(d,1,n)}, see Proposition \ref{PropInductionHypothesisG(d,1,n)}), we get terms of the three following forms:
\begin{itemize}
\item $s_{n-1} \cdots s_2 s_n \cdots s_3 z^c s_3 \cdots s_{i+1}$ with $0 \leq c \leq d-1$,
\item $s_{n-1} \cdots s_2 s_n \cdots s_3 z^c s_2 z^{c'}s_3 \cdots s_{i+1}$ with $0 \leq c \leq d-1$ and $0 \leq c' \leq d-1$,
\item $s_{n-1} \cdots s_2 s_n \cdots s_3 z^cs_2z^{c'}s_2s_3 \cdots s_{i+1}$ with $0 \leq c \leq d-1$ and $1 \leq c' \leq d-1$.
\end{itemize}

The first term is equal to $s_{n-1} \cdots s_2 z^c s_n \cdots s_4s_3^2s_4 \cdots s_{i+1}$ with $2 \leq i \leq n-1$. Thus, by the proof of Lemma \ref{LmScholie11}, it belongs to Span$(S_{n-1}^{*}\Lambda_n)$.

The second term is equal to $s_{n-1} \cdots s_2 z^c s_n \cdots s_3s_2s_3 \cdots s_{i+1}z^{c'}$. We have\\ $s_n \cdots s_3s_2s_3 \cdots s_{i+1}$ is equal to $s_2 \cdots s_i s_n \cdots s_2$. This is done in the first case of the proof of Lemma \ref{Lemma77}. Hence we get $s_{n-1} \cdots s_2 z^c s_2 \cdots s_i s_n \cdots s_2 z^{c'}$ which belongs to Span$(S_{n-1}^{*}\Lambda_n)$.

The third term is equal to $s_{n-1} \cdots s_2 z^c s_n \cdots s_2z^{c'}s_2 \cdots s_{i+1}$. Hence it belongs to Span$(S_{n-1}^{*}\Lambda_n)$.

\end{proof}

\begin{lemma}\label{Lemma1010}

If $a_{n-1} = s_{n-1} \cdots s_2 z^k s_2 \cdots s_i$ with $1 \leq k \leq d-1$, $2 \leq i \leq n-1$ and $a_n = s_n \cdots s_2 z^ls_2 \cdots s_{i'}$ with $1 \leq l \leq d-1$, $2 \leq i' \leq n$, then $s_n(a_{n-1}a_n)$ belongs to \emph{Span}$(S_{n-1}^{*}\Lambda_n)$.

\end{lemma}

\begin{proof}

According to the proof of the previous lemma, we have to prove that the following three terms belong to Span$(S_{n-1}^{*}\Lambda_n)$:
\begin{itemize}
\item $s_{n-1} \cdots s_2 z^cs_n \cdots s_4s_3^2s_4 \cdots s_{i+1}(s_2 \cdots s_{i'})$ for $0 \leq c \leq d-1$,
\item $s_{n-1} \cdots s_2 z^cs_2 \cdots s_i s_n \cdots s_2 z^{c'}(s_2 \cdots s_{i'})$ for $0 \leq c \leq d-1$ and $0 \leq c' \leq d-1$,
\item $s_{n-1} \cdots s_2 z^c s_n \cdots s_2 z^{c'} s_2 \cdots s_{i+1}(s_2 \cdots s_{i'})$ for $0 \leq c \leq d-1$ and $1 \leq c' \leq d-1$.
\end{itemize}

The first case is similar to Case 4 in the proof of Lemma \ref{Lemma10}.

For the second term, if $c' \neq 0$, it belongs to Span$(S_{n-1}^{*}\Lambda_n)$ and if $c' = 0$, by the computation in the proof of Lemma \ref{LmScholie11}, it also belongs to Span$(S_{n-1}^{*}\Lambda_n)$.

For the third term, we apply the same technique as in Case 2 of the proof of Lemma \ref{Lemma10}. We get that $s_n \cdots s_2 z^{c'}s_2 \cdots s_{i+1} (s_2 \cdots s_{i'})$ is equal to:\\
If $i' \leq i$, it is equal to $s_2s_3 \cdots s_{i'}s_n \cdots s_2 z^{c'}s_2 \cdots s_{i+1}$.\\
If $i' > i$, it is equal to $s_2s_3 \cdots s_{i'-1}s_n \cdots s_2 z^{c'}s_2 \cdots s_i$.\\
It follows that $s_{n-1} \cdots s_2 z^c s_n \cdots s_2 z^{c'} s_2 \cdots s_{i+1}(s_2 \cdots s_{i'})$ for $0 \leq c \leq d-1$ and $1 \leq c' \leq d-1$ belongs to Span$(S_{n-1}^{*}\Lambda_n)$.
 
\end{proof}

\begin{remark}

We remark that for every $d$ and $n$ at least equal to $2$, our basis never coincides with the Ariki-Koike basis \cite{ArikiKoikeHecke} as illustrated by the following example. Consider the element $s_2zs_2^2 = s_2zs_2.s_2$ which belongs to the Ariki-Koike basis. In our basis, it is equal to the linear combination $as_2zs_2 + s_2z$, where $s_2zs_2$ and $s_2z$ are two distinct elements of our basis.

\end{remark}

\section*{Acknowledgements}

Most of the results of this paper have been discovered during my PhD at Universit\'e de Caen Normandie. I would like to thank my PhD supervisors Eddy Godelle and Ivan Marin for the numerous stimulating discussions. I am also grateful to Barbara Baumeister for her helpful comments and for giving me the time to write this paper at the beginning of my postdoc (DFG project: BA 2200/5-1) at Universität Bielefeld.

\bibliographystyle{plain}
\bibliography{BasesHeckeV1}

\end{document}